\documentclass{article}

\PassOptionsToPackage{square,sort,comma,numbers}{natbib}
\usepackage[preprint]{neurips_2019}

\usepackage{amssymb, amsmath, amsthm}
\usepackage{dsfont, mathrsfs, color, bm, enumitem, url, xr}
\usepackage{subfigure, graphicx, transparent, booktabs} 
\usepackage{multirow, titlecaps, accents, makebox, epstopdf, mathtools}
\usepackage{caption}
\graphicspath{{graph/}}

\usepackage{algorithm, algorithmic}
\renewcommand{\algorithmicrequire}{\textbf{Input:}} 
\renewcommand{\algorithmicensure}{\textbf{Output:}}

\theoremstyle{definition}
\newtheorem{theorem}{Theorem}[section]
\newtheorem{definition}[theorem]{Definition}

\newtheorem{proposition}[theorem]{Proposition}
\newtheorem{corollary}[theorem]{Corollary}
\newtheorem{lemma}[theorem]{Lemma}
\newtheorem{remark}[theorem]{Remark}

\usepackage[utf8]{inputenc}    
\usepackage[T1]{fontenc}       
\usepackage{url}               
\usepackage{booktabs}          
\usepackage{amsfonts}          
\usepackage{nicefrac}          
\usepackage{microtype}         
\makeatletter
\def\munderbar#1{\underline{\sbox\tw@{$#1$}\dp\tw@\z@\box\tw@}}
\makeatother

\newcommand{\be}{\begin{equation}}
\newcommand{\ee}{\end{equation}}
\newcommand{\bea}{\begin{equation*}\begin{aligned}}
\newcommand{\eea}{\end{aligned}\end{equation*}}

\newcommand{\R}{\mathbb{R}}

\newcommand{\Max}{\max\limits_}
\newcommand{\Min}{\min\limits_}
\newcommand{\Sup}{\sup\limits_}
\newcommand{\Inf}{\inf\limits_}
\newcommand{\Tr}[1]{\Trace \big( #1 \big)}

\newcommand{\wh}{\hat}
\DeclareMathOperator{\grad}{grad\,}

\newcommand{\mc}{\mathcal}
\newcommand{\mbb}{\mathbb}
\newcommand{\inner}[2]{\big \langle #1, #2 \big \rangle }

\newcommand{\cov}{\Sigma} 
\newcommand{\covsa}{\wh{\cov}}

\newcommand{\Ambi}{\mc P} 
\DeclareMathOperator{\Trace}{Tr}

\newcommand{\PSD}{\mathbb{S}_{+}} 
\newcommand{\PD}{\mathbb{S}_{++}} 
\newcommand{\Let}{\triangleq}
\newcommand{\opt}{^\star}

\newcommand{\m}{\mu}
\newcommand{\msa}{\wh{\mu}}
\newcommand{\p}{\mbb P}
\newcommand{\psa}{\wh{\p}}
\newcommand{\U}{\mc U}

\newcommand{\half}{\frac{1}{2}}

\newcommand{\Proj}{\text{Proj}}

\newcommand{\N}{\mc N}

\newcommand{\radius}{\rho}

\DeclareMathOperator{\FR}{FR}
\DeclareMathOperator{\KL}{KL}
\DeclareMathOperator{\Exp}{Exp}
\newcommand{\B}{\mc B}

\newcommand{\dualvar}{\gamma}

\newcommand{\Z}{Z}

\newcommand{\Chol}{\Lambda}

\newcommand{\eig}{\lambda}

\newcommand{\ie}{{\it i.e.}}

\usepackage{enumitem}
\usepackage{wrapfig}
\author{
	Viet Anh Nguyen \qquad Soroosh Shafieezadeh-Abadeh\\
	\'Ecole Polytechnique F\'ed\'erale de Lausanne, Switzerland \\
	\texttt{ \{viet-anh.nguyen, soroosh.shafiee\}@epfl.ch } 
	\AND
	Man-Chung Yue \\
	The Hong Kong Polytechnic University, Hong Kong \\
	\texttt{manchung.yue@polyu.edu.hk} 
	\AND
	Daniel Kuhn \\
	\'Ecole Polytechnique F\'ed\'erale de Lausanne, Switzerland \\
	\texttt{ daniel.kuhn@epfl.ch } 
	\AND
	Wolfram Wiesemann \\
	Imperial College Business School, United Kingdom\\
	\texttt{ww@imperial.ac.uk} 
}

\title{Calculating Optimistic Likelihoods \\ Using (Geodesically) Convex Optimization}

\begin{document}
	
	\maketitle
	
	\begin{abstract}
		A fundamental problem arising in many areas of machine learning is the evaluation of the likelihood of a given observation under different nominal distributions. Frequently, these nominal distributions are themselves estimated from data, which makes them susceptible to estimation errors. We thus propose to replace each nominal distribution with an ambiguity set containing all distributions in its vicinity and to evaluate an \emph{optimistic likelihood}, that is, the maximum of the likelihood over all distributions in the ambiguity set. When the proximity of distributions is quantified by the Fisher-Rao distance or the Kullback-Leibler divergence, the emerging optimistic likelihoods can be computed efficiently using either geodesic or standard convex optimization techniques. We showcase the advantages of working with optimistic likelihoods on a classification problem using synthetic as well as empirical data. 
	\end{abstract}
	
	\section{Introduction}
	
	Assume that a set of i.i.d.~data points $x_1^M \Let x_1, \hdots, x_M \in \R^n$ is generated from one of several Gaussian distributions $\p_c$, $c \in \mathcal{C}$ with $| \mathcal{C} | < \infty$. To determine the distribution $\p_{c^\star}$, $c^\star \in \mathcal{C}$, under which $x_1^M$ has the highest likelihood $\ell(x_1^M, \p_c)$ across all $\p_c$, $c \in \mc C$, we can solve the problem
	\begin{equation}\label{eq:classical_MLE}
	c^\star \in \mathop{\arg \max}_{c \in \mc C} ~\left\{ \ell(x_1^M, \p_c)
	\Let -\frac{1}{M} \sum_{m=1}^M (x_m-\mu_c)^\top \cov^{-1}_c (x_m - \mu_c) - \log \det  \cov_c \right\} ,
	\end{equation}
	where $\mu_c$ and $\Sigma_c$ denote the means and covariance matrices that unambiguously characterize the distributions $\p_c$, $c \in \mc C$, and the log-likelihood function $\ell(x_1^M, \p_c)$ quantifies the (logarithm of the) relative probability of observing $x_1^M$ under the Gaussian distribution $\p_c$. Problem~\eqref{eq:classical_MLE} naturally arises in various machine learning applications. In quadratic discriminant analysis, for example, $x_1^M$ denotes the input values of data samples whose categorical outputs $y_1, \hdots, y_M \in \mathcal{C}$ are to be predicted based on the class-conditional distributions $\mathbb{P}_c$, $c \in \mathcal{C}$ \cite{mclachlan2004discriminant}. Likewise, in Bayesian inference with synthetic likelihoods, a Bayesian belief about the models $\p_c$, $c \in \mathcal{C}$, assumed to be Gaussian for computational tractability, is performed based on an observation $x_1^M$ \cite{price2018bayesian, wood2010statistical}. Problem~\eqref{eq:classical_MLE} also arises in likelihood-ratio tests where the null hypothesis `$x_1^M$ is generated by a distribution $\p_c$, $c \in \mathcal{C}_0$' is compared with the alternative hypothesis `$x_1^M$ is generated by a distribution $\p_c$, $c \in \mathcal{C}_1$' \cite{cox1961tests, cox2013return}.
	
	In practice, the parameters $(\mu_c, \Sigma_c)$ of the candidate distributions $\mathbb{P}_c$, $c \in \mathcal{C}$, are typically not known and need to be estimated from data. In quadratic discriminant analysis, for example, it is common to replace the means $\mu_c$ and covariance matrices $\Sigma_c$ with their empirical counterparts $\hat{\mu}_c$ and $\hat{\Sigma}_c$ that are estimated from data. Similarly, the rival model distributions $\p_c$, $c \in \mathcal{C}$, in Bayesian inference with synthetic likelihoods are Gaussian estimates derived from (costly) sampling processes. Unfortunately, problem~\eqref{eq:classical_MLE} is highly sensitive to misspecification of the candidate distributions $\p_c$. To combat this problem, we propose to replace the likelihood function in~\eqref{eq:classical_MLE} with the \emph{optimistic likelihood}
	\begin{equation}\label{eq:optimistic_MLE}
	\Max{\p \in \Ambi_c} ~ \ell(x_1^M, \p)
	\;\; \text{with} \;\;
	\Ambi_c = \left\{ \mathbb{P} \in \mathcal{M}: \varphi(\psa_c, \p) \leq \radius_c \right\}  ,
	\end{equation}
	where $\mathcal{M}$ is the set of all non-degenerate Gaussian distributions on $\R^n$, $\varphi$ is a dissimilarity measure satisfying $\varphi(\p, \p)=0$ for all $\p \in \mathcal{M}$, and $\rho_c \in \R_+$ are the radii of the ambiguity sets $\mathcal{P}_c$. Problem~\eqref{eq:optimistic_MLE} assumes that the true candidate distributions $\p_c$ are unknown but close to the nominal distributions $\psa_c$ that are estimated from the training data. In contrast to the log-likelihood $\ell(x_1^M, \p_c)$ that is maximized in problem~\eqref{eq:classical_MLE}, the optimistic likelihood~\eqref{eq:optimistic_MLE} is of interest in its own right. A common problem in constrained likelihood estimation, for example, is to determine a Gaussian distribution $\p^\star \sim (\mu^\star, \Sigma^\star)$ that is close to a nominal distribution $\p^0 \sim (\mu^0, \Sigma^0)$ reflecting the available prior information such that $x_1^M$ has high likelihood under $\p^\star$~\cite{schoenberg1997constrained}. This task is an instance of the optimistic likelihood evaluation problem~\eqref{eq:optimistic_MLE} with a suitably chosen dissimilarity measure $\varphi$.

	Of crucial importance in the generalized likelihood problem~\eqref{eq:optimistic_MLE} is the choice of the dissimilarity measure $\varphi$ as it impacts both the statistical properties as well as the computational complexity of the estimation procedure. A natural choice appears to be the Wasserstein distance, which has recently been popularized in the field of optimal transport \cite{takatsu2011wasserstein, villani2008optimal}. 
		The Wasserstein distance on the space of Gaussian distributions is a Riemannian distance, that is, the distance corresponding the curvilinear geometry on the set of Gaussian distributions induced by the Wasserstein distance, as opposed to the usual distance obtained by treating it as a subset of the space of symmetric matrices.
	However, since the Wasserstein manifold has a non-negative sectional curvature~\cite{takatsu2011wasserstein}, calculating the associated optimistic likelihood~\eqref{eq:optimistic_MLE} appears to be computationally intractable. Instead, we study the optimistic likelihood under the Fisher-Rao (FR) distance, which is commonly used in signal and image processing \cite{pennec2006riemannian, arnaudon2013riemannian} as well as computer vision \cite{jayasumana2013kernel, tuzel2008pedestrian}. The FR distance is also a Riemannian metric, and it enjoys many attractive statistical properties that we review in Section~\ref{sect:FR} of this paper. Most importantly, the FR distance has a non-positive sectional curvature, which implies that the optimistic likelihood~\eqref{eq:optimistic_MLE} reduces to the solution of a geodesically convex optimization problem that is amenable to an efficient solution~\cite{bento2017iteration, ref:Sra-2015, tripuraneni2018averaging, zhang2016riemannian, ref:Zhang-2016, zhang2018estimate}. We also study problem~\eqref{eq:optimistic_MLE} under the Kullback–Leibler (KL) divergence (or relative entropy), which is intimately related to the FR metric. While the KL divergence lacks some of the desirable statistical features of the FR metric, we will show that it gives rise to optimistic likelihoods that can be evaluated in quasi-closed form by reduction to a one dimensional problem.
	
	While this paper focuses on the parametric approximation of the likelihood where $\mbb P$ belongs to the family of Gaussian distributions, we emphasize that the optimistic likelihood approach can also be utilized in a \textit{non-}parametric setting~\cite{ref:nguyen2019distributionally}.
	
	The contributions of this paper may be summarized as follows.
	\begin{enumerate}[leftmargin=*]
		\item We show that for Fisher-Rao ambiguity sets, the optimistic likelihood~\eqref{eq:optimistic_MLE} reduces to a geodesically convex problem and is hence amenable to an efficient solution via a Riemannian gradient descent algorithm. We analyze the optimality as well as the convergence of the resulting algorithm.
		\item We show that for Kullback-Leibler ambiguity sets, the optimistic likelihood~\eqref{eq:optimistic_MLE} can be evaluated in quasi-closed form by reduction to a one dimensional convex optimization problem.
		\item We evaluate the numerical performance of our optimistic likelihoods on a classification problem with artificially generated as well as standard benchmark instances.
	\end{enumerate}
	
	Our optimistic likelihoods follow a broader optimization paradigm that exercises optimism in the face of ambiguity. This strategy has been shown to perform well, among others, in multi-armed bandit problems and Bayesian optimization, where the Upper Confidence Bound algorithm takes decisions based on optimistic estimates of the reward \cite{brochu2010tutorial, bubeck2012regret, munos2014bandits, srinivas2010gaussian}. Optimistic optimization has also been successfully applied in support vector machines \cite{bi2005support}, and it closely relates to sparsity inducing non-convex regularization schemes \cite{norton2017optimistic}.
	
	The remainder of the paper proceeds as follows. We study the optimistic likelihood~\eqref{eq:optimistic_MLE} under FR and KL ambiguity sets in Sections~\ref{sect:FR} and~\ref{sect:KL}, respectively. We test our theoretical findings in the context of a classification problem, and we report on numerical experiments in Section~\ref{sect:awesome_numericals}. Supplementary material and all proofs are provided in the online companion.
	
	\textbf{Notation.} Throughout this paper, $\mbb S^n$, $\PSD^n$ and $\PD^n$ denote the spaces of $n$-dimensional symmetric, symmetric positive semi-definite and symmetric positive definite matrices, respectively. For any $A \in \R^{n\times n}$, the trace of $A$ is defined as $\Tr{A} = \sum_{i=1}^n A_{ii}$. For any $A \in \mathbb{S}^n$, $\eig_{\min}(A)$ and $\eig_{\max}(A)$ denote the minimum and maximum eigenvalues of $A$, respectively. The base of $\log(\cdot)$ is $e$.


\section{Optimistic Likelihood Problems under the FR Distance}
\label{sect:FR}

Consider a family of distributions with density functions $p_\theta(x)$, where the parameter $\theta$ ranges over a finite-dimensional smooth manifold $\Theta$. At each point $\theta \in \Theta$, the Fisher information matrix $I_\theta=\mathbb E_x [\nabla_\theta \log(p_\theta(x)) \nabla_\theta \log(p_\theta(x))^\top|\theta]$ defines an inner product $\langle \cdot,\cdot \rangle_\theta$ on the tangent space $T_\theta \Theta$ by $\langle \zeta_1, \zeta_2 \rangle_\theta = \zeta_1^T I_\theta \zeta_2$ for $\zeta_1,\zeta_2 \in T_\theta \Theta$. The family of inner products $\{\langle \cdot, \cdot \rangle_\theta\}_{\theta\in \Theta}$ on the tangent spaces then defines a Riemannian metric, called the FR metric. The FR distance on $\Theta$ is the geodesic distance associated with the FR metric, \ie, the FR distance between the two points $\theta_0, \theta_1\in\Theta$ is  $$ d (\theta_0, \theta_1) = \inf_{\gamma} \int_0^1 \sqrt{\langle \gamma'(t), \gamma'(t) \rangle_{\gamma(t)}}dt,$$
where the infimum is taken over all smooth curves $\gamma:[0,1]\rightarrow \Theta$ with $\gamma(0) = \theta_0$ and $\gamma(1) = \theta_1$.
Any curve $\gamma$ attaining the infimum is said to be a geodesic from $\theta_0$ to $\theta_1$.
The FR metric represents a natural distance measure for parametric families of probability distributions as it is invariant under transformations on the data space (the $x$ space) by a class of statistically important mappings, and it is the unique (up to a scaling) Riemannian metric enjoying such a property, see~\cite{campbell1986extended,cencov2000statistical,bauer2016uniqueness}.

Since the covariance matrix is more difficult to estimate than the mean (see Appendix~\ref{sect:appendix:justification}), we focus here on the family of all Gaussian distributions with a fixed mean vector $\msa\in\mathbb R^n$. These distributions are parameterized by $\theta=\cov$, that is, the covariance matrix. The parameter manifold is thus given by $\Theta=\PD^n$. On this manifold, the FR distance is available in closed form.\footnote{We can also handle the case where the covariance matrix is fixed but the mean is subject to ambiguity, see Appendix~\ref{sect:mean}. However, as there is no closed-form expression for the FR distance between two generic Gaussian distributions, we cannot handle the case where both the mean and the covariance matrix are subject to ambiguity.} 

\begin{proposition}[FR distance for Gaussian distributions \cite{atkinson1981rao}]
	If~$\N(\msa, \cov_0)$ and $ \N(\msa, \cov_1)$ are Gaussian distributions with identical mean $\msa \in \R^n$ and covariance matrices $\cov_0, \cov_1 \in \PD^n$, we have
	\begin{equation} \label{eq:FR:cov}
	d(\cov_0, \cov_1) = \frac{1}{\sqrt{2}} \left\| \log(\cov_1^{-\half} \cov_0 \cov_1^{-\half}) \right\|_F ,
	\end{equation}
	where $\log(\cdot)$ represents the matrix logarithm, and $\| \cdot \|_F$ stands for the Frobenius norm.
\end{proposition}
The distance~$d(\cdot,\cdot)$ is invariant under \emph{inversions} and \emph{congruent transformations} of the input parameters~\cite[Proposition~1]{said2017riemannian}, \ie, for any $\covsa,\cov \in \PD^n$ and invertible matrix $A \in \mathbb{R}^{n \times n}$, we have
\begin{align}
d(\covsa^{-1}, \cov^{-1}) &= d(\covsa, \cov) \quad \label{eq:invar_1}\\
\text{and}\quad d(A \covsa A^\top, A\cov A^\top) &= d(\covsa, \cov).\label{eq:invar_2}
\end{align}
%
By the inversion invariance~\eqref{eq:invar_1}, the distance $d(\cdot, \cdot)$ is independent of whether we use the covariance matrix $\cov$ or the precision matrix $\cov^{-1}$ to parametrize normal distributions. 
Note that if $x_1\sim \mathcal{N}(\mu, \cov_1)$ and $x_2 \sim\mathcal{N}(\mu, \cov_2)$, then $Ax_1 + b \sim \mathcal{N}(A\mu + b, A\cov_1 A^\top)$ and $Ax_2 + b \sim \mathcal{N}(A\mu + b, A\cov_2 A^\top)$. By the congruence invariance~\eqref{eq:invar_2}, the distance $d(\cdot, \cdot)$ thus remains unchanged under affine transformations $x \rightarrow Ax + b$. Remarkably, the invariance property~\eqref{eq:invar_2} uniquely characterizes the distance $d(\cdot, \cdot)$. More precisely, any Riemannian distance satisfying the invariance property~\eqref{eq:invar_2} coincides (up to a scaling) with the distance $d(\cdot, \cdot)$, see, for example, \cite[Section 3]{savage1982space} and \cite[Section 2]{bonnabel2009riemannian}.

We now study the optimistic likelihood problem \eqref{eq:optimistic_MLE}, where the FR distance is used as the dissimilarity measure. Given a data batch $ x_1^M$ and a radius $\rho>0$, the optimistic likelihood problem reduces~to
\begin{equation} \label{eq:FR}
\Min{\cov\in \B^{\FR}} ~ L(\cov),\quad \text{where} \quad \left\{ \begin{array}{l}
L(\cov) \Let \inner{S}{\cov^{-1}} + \log \det \cov, \\ 
\B^{\FR}\Let \{ \Sigma \in \mathbb{S}_{++}^n : d(\cov, \covsa) \le \radius \},\end{array}\right.
\end{equation}
and $S=  M^{-1} \sum_{m=1}^M (x_m - \wh \m)(x_m - \wh \m)^\top$ stands for the sample covariance matrix.


We next prove that problem~\eqref{eq:FR} is solvable, which justifies the use of the minimization operator.
\begin{lemma}\label{lemma:opt_exist}
The optimal value of problem~\eqref{eq:FR} is finite and is attained by some $\cov^\star\in \B^{\FR}$.
\end{lemma}

Even though the objective function of~\eqref{eq:FR} involves a concave log-det term, it can be shown to be convex over the region $0 \prec \cov \preceq 2S$ \citep[Exercise~7.4]{ref:Boyd-2004}. However, in practice $S$ may be singular, in which case this region becomes empty. Maximum likelihood estimation problems akin to~\eqref{eq:FR} are often reparameterized in terms of the precision matrix $X = \cov^{-1}$. In this case,~\eqref{eq:FR} becomes
\begin{equation*}
\min \left\{  \inner{S}{X} - \log \det  X : \, X \in \PD^n, \, \| \log(X^\half \covsa X^\half)\|_F \leq \sqrt{2}\radius \right\}.
\end{equation*}
Even though this reparameterization convexifies the objective, it renders the feasible set non-convex.
	

\subsection{Geodesic Convexity of the Optimistic Likelihood Problem}
As problem~\eqref{eq:FR} cannot be addressed with standard methods from convex optimization, we re-interpret it as a constrained minimization problem on the Riemannian manifold $\Theta=\PD^n$ endowed with the FR metric. The key advantage of this approach is that we can show problem~\eqref{eq:FR} to be \emph{geodesically convex}. Geodesic convexity generalizes the usual notion of convexity in Euclidean spaces to Riemannian manifolds. We can thus solve problem~\eqref{eq:FR} via algorithms from geodesically convex optimization, which inherit many benefits of the standard algorithms of convex optimization in Euclidean spaces. 

The Riemannian manifold $\Theta=\PD^n$ endowed with the FR metric is in fact a Hadamard manifold, that is, a complete simply connected Riemannian manifold with non-positive sectional curvature, see~\cite[Theorem XII 1.2]{lang2012fundamentals}. Thus, any two points are connected by a {\em unique} geodesic~\cite{ref:Bridson-1999}. By \cite[Theorem 6.1.6]{bhatia2009positive}, for $\cov_0,\cov_1 \in \PD^n$, the unique geodesic $\gamma: [0,1] \rightarrow \mathbb{S}^n_{++}$ from $\cov_0$ to $\cov_1$ is given by 
\begin{equation} \label{eq:gamma}
\gamma(t) = \cov_0^\half \big( \cov_0^{-\half} \cov_1 \cov_0^{-\half} \big)^{t} \cov_0^\half.
\end{equation}
We are now ready to give precise definitions of geodesically convex sets and functions on Hadamard manifolds. We emphasize that these definitions would be more subtle for general Riemannian manifolds, which can have several geodesics between two points. 

\begin{definition}[Geodesically convex set]
	A set $\U \subseteq \mathbb{S}^n_{++}$ is said to be geodesically convex if for all $\cov_0, \cov_1 \in \U$, the image of the unique geodesic from $\cov_0$ to $\cov_1$ is contained in $\U$, \ie, $\gamma([0,1]) \subseteq \mathcal{U}$.
\end{definition}
\begin{definition}[Geodesically convex function]\label{def:g-convex}
	A function $f: \PD^n \to \R$ is said to be geodesically convex if for all $\cov_0, \cov_1 \in \mathbb{S}^n_{++}$, the unique geodesic $\gamma$ from $\cov_0$ to $\cov_1$ satisfies
	$f(\gamma(t)) \leq (1-t) f(\cov_0) + t f(\cov_1) \  \forall t \in [0,1].$
\end{definition}

In order to prove that~\eqref{eq:FR} is a geodesically convex optimization problem, we need to establish the geodesic convexity of the feasible region $\mathcal{B}^{\FR}$ and the loss function $L(\cdot)$.
Note that, in stark contrast to Euclidean geometry, a geodesic ball on a general manifold may not be geodesically convex.\footnote{For example, consider the circle $S^1\Let \{x\in \mathbb{R}^2: \|x\|_2 = 1 \}$ which is a 1-dimensional manifold. Any major arc is a geodesic ball but {\em not} a geodesically convex subset of~$S^1$.} 
\begin{theorem}[Geodesic convexity of problem~\eqref{eq:FR}] \label{thm:convex}
	For any $\covsa \in \PD^n$, $S \in \PSD^n$ and $\rho \in \R_+$, $\B^{\FR}$ is a geodesically convex set, and $L(\cdot)$ is a geodesically convex function over $\PD^n$. 
\end{theorem}

Theorem~\ref{thm:convex} establishes that the optimistic likelihood problem~\eqref{eq:FR}, which is non-convex with respect to the usual Euclidean geometry on the embedding space $\mathbb R^{n\times n}$, is actually convex with respect to the Riemannian geometry on $\PD^n$ induced by the FR metric.


\newpage
\subsection{Projected Geodesic Gradient Descent Algorithm}
\begin{wrapfigure}{L}{0.47\textwidth}
	\centering
		\includegraphics[width=0.44\textwidth]{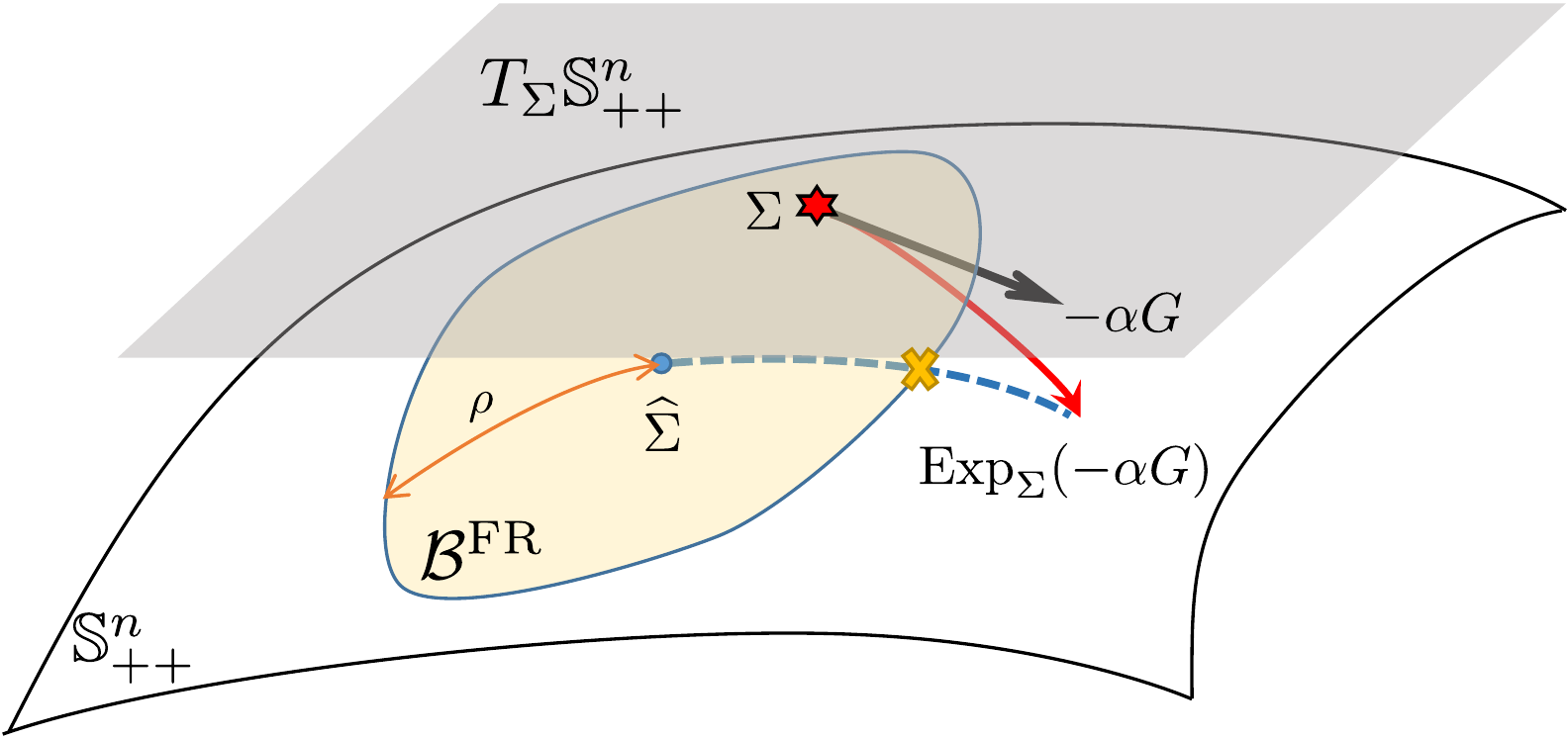}
	\caption{Visualization of the FR ball $\B^{\FR}$ (yellow set) within the manifold $\PD^n$ (white set).}
	\label{fig:FR_ball}
\end{wrapfigure}
In the same way as the convexity of a standard constrained optimization problem can be exploited to find a global minimizer via a projected gradient descent algorithm, the geodesic convexity of problem~\eqref{eq:FR} can be exploited to find a global minimizer by using a \emph{projected geodesic gradient descent} algorithm of the type described in~\cite{ref:Zhang-2016}. The mechanics of a generic iteration are visualized in Figure~\ref{fig:FR_ball}. As in any gradient descent method, given the current iterate $\cov$, we first need to compute the direction along which the objective function $L$ decreases fastest. In the context of optimization on manifolds, this direction corresponds to the negative Riemannian gradient $-G$ at point~ $\cov$, which belongs to the tangent space $T_{\cov}\PD^n\simeq\mathbb S^n$. 
Unfortunately, the curve $\gamma(\alpha)=\cov - \alpha G$ fails to be a geodesic and will eventually leave the manifold for sufficiently large step sizes $\alpha$. This prompts us to construct the (unique) geodesic that emanates from point $\cov$ with initial velocity $-G$. Formally, this geodesic can be represented as $\gamma(\alpha)=\Exp_{\cov}(-\alpha G)$, where $\Exp_{\cov}(\cdot)$ denotes the \emph{exponential map} at~$\cov$. 
As we will see below, this geodesic (visualized by the red curve) remains within the manifold for any $\alpha>0$ but may eventually leave the feasible region $\B^{\FR}$. If this happens for the chosen step size $\alpha$, we project $\Exp_{\cov}(-\alpha G)$ back onto the feasible region, that is,  we map it to its closest point in $\B^{\FR}$ with respect to the FR distance (visualized by the yellow cross). Denoting this FR projection by $\Proj_{\B^{\FR}} (\cdot)$, the next iterate of the projected geodesic gradient descent algorithm can thus be expressed as $\cov^+=\Proj_{\B^{\FR}} (\Exp_{\cov} (-\alpha G))$. 

Starting from $\cov_1=\covsa$, the proposed algorithm constructs $K$ iterates $\{\cov_k\}_{k=1}^K$ via the above recursion. As in~\cite{ref:Zhang-2016}, the algorithm also constructs a second sequence $\{\bar{\cov}_k\}_{k=1}^K$ of feasible covariance matrices with $\bar\cov_1=\covsa$ and $\bar\cov_{k+1}=\bar\gamma(1/(k+1))$ for $k=1,\ldots, K-1$, where $\bar\gamma(t)$ represents the geodesic~\eqref{eq:gamma} connecting $\bar\cov_k$ with $\cov_{k+1}$. Thus, $\bar{\cov}_{k+1}$ is defined as a \emph{geodesic convex combination} of $\bar{\cov}_k$ and $\cov_{k+1}$. 
A precise description of the proposed algorithm in pseudocode is provided in Algorithm~\ref{alg:projected:geodesic}.


In the following we show that the Riemannian gradient, the exponential map $\Exp_{\cov}(\cdot)$ as well as the projection $\Proj_{\B^{\FR}} (\cdot)$ can all be evaluated in closed form in $\mc O(n^3)$.


\begin{algorithm}[t]
	\caption{Projected Geodesic Gradient Descent Algorithm}
	\label{alg:projected:geodesic}
	\begin{algorithmic}
		\STATE {\bfseries Input:} $\covsa \succ 0$, $\rho > 0$, $S \succeq 0$, $K\in\mathbb N$, $\{\alpha_k\}_{k=1}^K\subseteq \mathbb R_{++}$
		\STATE {\bfseries Initialization:} Set $\cov_1 \leftarrow \covsa$, $\bar{\cov}_1 \leftarrow \covsa$
		\FOR{$k=1, 2,  \ldots, K-1$}
		\STATE Compute the Riemannian gradient at $\cov_k$: $ G_k \leftarrow   2(\cov_k -  S )$
		\STATE Perform a gradient descent step using the exponential map:
		\begin{align*}
			\textstyle \cov_{k+\half} \leftarrow \Exp_{\cov_k} (-\alpha_k G_k) = \cov_k^\half \exp\big(-\alpha_k \cov_k^{-\half} G_k \cov_k^{-\half} \big) \cov_k^\half
		\end{align*}
		\STATE Project $\cov_{k+\half}$ onto $\B^{\FR}$:
		$ \cov_{k+1} \leftarrow \Proj_{\B^{\FR}}(\cov_{k+\half})$
		\STATE Compute the new iterate by interpolation: $\bar{\cov}_{k+1}  \leftarrow \textstyle \Exp_{\bar{\cov}_k} \big( \frac{1}{k+1} \Exp_{\bar{\cov}_k}^{-1} \left( \cov_{k+1} \right) \big)  
		$
		\ENDFOR 
		\STATE{\bfseries Output:} Report the last iterate $\bar{\cov}_K$ as an approximate solution
	\end{algorithmic}
\end{algorithm}

By \cite[Page 362]{atkinson1981rao}, the FR metric on the tangent space $T_{\cov}\PD^n$ at $\cov \in \mathbb{S}_{++}^n$ can be re-expressed as
\begin{equation}\label{eq:metric}
\inner{\Omega_1}{\Omega_2}_\Sigma \Let \frac{1}{2}\Tr{\Omega_1 \Sigma^{-1} \Omega_2 \Sigma^{-1}} \quad \forall \Omega_1, \Omega_2 \in T_\Sigma \PD^n.
\end{equation}
Using \eqref{eq:metric} and \cite[Equation 3.32]{AbsMahSep2008}, the Riemannian gradient $G=\text{grad}\,L$ of the objective function $L(\cdot )$ at $\cov$ can be computed from the Euclidean gradient $\nabla L(\cov)$ as
\begin{equation}\label{eq:grad}
	\grad L (\cov)  = 2 \cov ( \nabla L(\cov) ) \cov = 2\cov(\cov^{-1} -\cov^{-1} S \cov^{-1} )\cov =2( \cov - S).
\end{equation}
Moreover, by~\cite[Equation (3.2)]{ref:Sra-2015}, the exponential map $\Exp_{\cov}: T_\cov\mathbb{S}_{++}^n \to \mathbb{S}_{++}^n$ at $\cov$ is given by
\begin{equation*}
	\Exp_{\cov}(G) = \cov^\half \exp( \cov^{-\half} G \cov^{-\half} )\cov^\half, \quad G\in  T_\cov\mathbb{S}_{++}^n \simeq \mathbb{S}^n,
\end{equation*}
where $\exp(\cdot)$ denotes the matrix exponential.
The inverse map $\Exp_\cov^{-1}: \mathbb{S}_{++}^n \to T_\cov \mathbb{S}_{++}^n$ satisfies
$$
	\Exp_\cov^{-1} (A) = \cov^\half \big( \log \cov^{-\half} A \cov^{-\half} \big) \cov^\half, \quad A \in \mathbb{S}_{++}^n.
$$
Finally, the projection $\Proj_{\B}(\cdot)$ onto $\mathcal{B}^{\FR}$ with respect to the FR distance is defined through
\begin{equation}\label{eq:proj0}
\Proj_{\B^{\FR}}(\cov') \Let \mathop{\arg\min}_{\cov \in \B^{\FR}} \; d(\cov, \cov'), \quad \cov' \in \mathbb{S}_{++}^n.
\end{equation}
The following lemma ensures that this projection is well-defined and admits a closed-form expression.
\begin{lemma}[Projection onto $\mathcal{B}^{\FR}$]
	\label{lemma:projection}
	For any $\cov' \in \PD^n$ with $d(\covsa, \cov') = \rho'$ the following hold.
	\begin{enumerate} [wide, labelwidth=!, labelindent=0pt, label=(\roman*)]
		\item\label{item:1} There $\arg\min$-mapping in~\eqref{eq:proj0} is a singleton, and thus $\Proj_{\B^{\FR}}(\cov')$ is well-defined. 
		\item\label{item:2} The projection of $\cov'$ onto $\B^{\FR}$ is given by
		\begin{equation}
		\label{eq:proj}
		\Proj_{\B^{\FR}}(\cov') = \begin{cases}
   \covsa^\half \big( \covsa^{-\half} \cov' \covsa^{-\half} \big)^{\frac{\radius}{\rho'}} \covsa^\half & \mbox{if } \rho' > \radius, \\
   \cov'& \mbox{otherwise.} 
		\end{cases}
		\end{equation}
	\end{enumerate}
\end{lemma}
By comparison with~\eqref{eq:gamma}, one easily verifies that $\Proj_{\B^{\FR}}(\cov')$ constitutes a geodesic convex combination between $\cov'$ and $\covsa$. Figure~\ref{fig:FR_ball} visualizes the geodesic from $\covsa$ to $\cov'$ by the blue dashed line. Therefore, the projection $\Proj_{\B^{\FR}}$ onto the FR ball $\B^{\FR}$ within $\PD^n$ endowed with the FR metric is constructed in a similar manner as the projection onto a Euclidean ball within a Euclidean space.

The following theorem asserts that Algorithm~\ref{alg:projected:geodesic} enjoys a sublinear convergence rate.
\begin{theorem}[Sublinear convergence rate]
	\label{thm:convergence}
	With a constant stepsize
	$$
		\alpha_k \equiv  2^{1/4} \sqrt{\rho \tanh(2\sqrt{2}\rho)}/(\Gamma \sqrt{K}),
	$$
	where $\Gamma \Let 2^{-1/2}\sqrt{n}\cdot e^{2\sqrt{2}\rho} \cdot\eig_{\min}^{-2}(\covsa) \cdot \max\{ |1 - e^{\sqrt{2}\rho} \eig^{-1}_{\min}(\hat{\Sigma}) \eig_{\max} (S)| , 1 \}$,
Algorithm~\ref{alg:projected:geodesic} satisfies
	\[
		L(\bar{\cov}_K) - L(\cov\opt) \leq  \frac{2^{\frac{7}{4}} \rho^{\frac{3}{2} }\Gamma}{\sqrt{K \tanh(2\sqrt{2}\rho)}}.
	\]
\end{theorem}

The proof of Theorem~\ref{thm:convergence} closely follows that of \cite[Theorem 9]{ref:Zhang-2016}. The difference is that~\cite[Theorem 9]{ref:Zhang-2016} requires the objective function to be Lipschitz continuous on $\PD^n$. Unfortunately, such an assumption is not satisfied by $L(\cdot)$. We circumvent this by proving that the Riemannian gradient of~$L(\cdot)$ is bounded uniformly on $\B^{\FR}$. 

Endeavors are currently underway to devise algorithms for minimizing a geodesically strongly convex objective function over a geodesically convex feasible set that offer a linear convergence guarantee, see, e.g., \cite[Theorem 15]{ref:Zhang-2016}. The next lemma shows that the objective function of problem~\eqref{eq:FR} is indeed geodesically smooth and geodesically strongly convex\footnote{The strong convexity and smoothness properties are defined in Definitions~\ref{def:strong-convexity} and~\ref{def:smoothness}, respectively.} whenever $S \succ 0$. This suggests that the empirical performance of Algorithm~\ref{alg:projected:geodesic} could be significantly better than the theoretical guarantee of Theorem~\ref{thm:convergence}. Indeed, our numerical results in Section~\ref{sect:awesome_numericals:the_convergence} confirm that if $S \succ 0$, then Algorithm~\ref{alg:projected:geodesic} displays a linear convergence rate. 

\begin{lemma}[Strong convexity and smoothness of $L(\cdot)$] \label{lemma:obj_function}
	The objective function $L(\cdot)$ of problem~\eqref{eq:FR} is geodesically $\beta$-smooth on $\B^{\FR}$ with $$ \beta = \frac{2 \lambda_{\max}(S)}{\lambda_{\min}(\widehat{\Sigma}) \exp(-\sqrt{2}\rho)} .$$ If $S \succ 0$, then $L(\cdot)$ is also geodesically $\sigma$-strongly convex on $\B^{\FR}$ with $$ \sigma = \frac{2 \lambda_{\min}(S)}{\lambda_{\max}(\widehat{\Sigma}) \exp(\sqrt{2}\rho)} . $$
	\end{lemma}

\begin{remark}
Problem~\eqref{eq:FR} could also be addressed with the algorithmic framework developed in~\cite{liu2019riemannianconstraints}. Due to space limitations, we leave this for future research.
\end{remark}
\section{Generalized Likelihood Estimation under the KL Divergence}
\label{sect:KL}

The KL divergence, which is widely used in information theory \cite[\S~2]{cover2006elements}, can be employed as an alternative dissimilarity measure in the optimistic likelihood problem \eqref{eq:optimistic_MLE}.
%
If both $\psa$ and $\p$ are Gaussian distributions, then the KL divergence from $\psa$ to $\p$ admits an analytical expression.

\begin{proposition}[KL divergence for Gaussian distributions] \label{prop:KL}
	For any $\msa\in \R^n$ and $\cov_0, \cov_1 \in \PD^n$, the KL divergence from $\p_0 = \N(\msa, \cov_0)$ to $\p_1 = \N(\msa, \cov_1)$ amounts to
	\begin{align*}
	\KL(\p_0 \parallel \p_1) =\frac{1}{2} \big( \Tr{\cov_1^{-1} \cov_0} + \log \det(\cov_1 \cov_0^{-1}) - n \big).
	\end{align*}
\end{proposition}
Unlike the FR distance, the KL divergence is not symmetric. Proposition~\ref{prop:KL} implies that if the KL divergence is used as the dissimilarity measure, then the optimistic likelihood problem \eqref{eq:optimistic_MLE} reduces~to
\begin{equation} \label{eq:KL:prob}
\begin{array}{cl}
\Min{\cov \succ 0} & \left\{ \Tr{S \cov^{-1}} + \log \det \cov: \Tr{\cov^{-1}\covsa} + \log \det (\cov\covsa^{-1}) - n \leq 2\radius \right\},
\end{array}
\end{equation}
where $S=  M^{-1} \sum_{m=1}^M (x_m - \wh \m)(x_m - \wh \m)^\top$ denotes again the sample covariance matrix.
Because of the concave log-det terms in the objective and the constraints, problem~\eqref{eq:KL:prob} is non-convex. By using the variable substitution $X\leftarrow \cov^{-1}$, however, problem~\eqref{eq:KL:prob} can be reduced to a univariate convex optimization problem and thereby solved in quasi-closed form.

\begin{theorem}
	\label{thm:KL}
	For any $\covsa \succ 0$ and $\radius > 0$, the optimal value of problem~\eqref{eq:KL:prob} amounts to
	\begin{align*}
		(1+\dualvar\opt)\Tr{S (S + \dualvar\opt\covsa)^{-1}}  + \log\det (S + \dualvar\opt \covsa) - n \log(1+ \dualvar\opt),
	\end{align*}
	where $\dualvar\opt$ is the unique optimal solution of the univariate convex optimization problem
	\begin{align}
	\Min{\dualvar > 0} \Big\{ \dualvar (2\radius + \log \det \covsa) + n(1 + \dualvar) \log(1+\dualvar)  - (1 + \dualvar) \log \det (S + \dualvar \covsa) \Big\}. \label{eq:KL:cov}
	\end{align}
\end{theorem}

Problem~\eqref{eq:KL:cov} can be solved efficiently using state-of-the-art first- or second-order methods, see Appendix~\ref{sect:appendix:derivatives}. However, in each iteration we still need to evaluate the determinant of a positive definite $n$-by-$n$ matrix, which requires $\mc O(n^{3})$ arithmetic operations. The following corollary shows that this computational burden can be alleviated when the sample covariance matrix $S$ has low rank.

\begin{corollary}[Singular sample covariance matrices] \label{corol:singular}
	If $S = \Chol \Chol^\top$ for some $\Chol \in \R^{n\times k}$ and $k\in\mathbb N$ with $k < n$, then problem~\eqref{eq:KL:cov} simplifies to
	\begin{align*}
	\Min{\dualvar > 0} \Big\{ 2\dualvar \radius + n(1+\dualvar) \log \left(1+\dualvar\right) - (n-k)(1+\dualvar) \log\dualvar -(1+\dualvar) \log \det (\dualvar I_k +  \Chol^\top \covsa^{-1} \Chol ) \Big\}. 
	\end{align*}
\end{corollary}

We will see that for classification problems the matrix $S$ has rank~1, in which case the log-det term in the above univariate convex program reduces to the scalar logarithm. In Appendix~\ref{sect:appendix:derivatives} we provide explicit first- and second-order derivatives of the objective of problem~\eqref{eq:KL:cov} and its simplification.

%
%
%

\section{Numerical Results}
\label{sect:awesome_numericals}

We investigate the empirical behavior of our projected geodesic gradient descent algorithm (Section~\ref{sect:awesome_numericals:the_convergence}) and the predictive power of our flexible discriminant rules (Section~\ref{sec:classification}). Our algorithm and all tests are implemented in Python, and the source code is available from~\url{https://github.com/sorooshafiee/Optimistic_Likelihoods}.

\subsection{Convergence Behavior of the Projected Geodesic Descent Algorithm}
\label{sect:awesome_numericals:the_convergence}

To study the empirical convergence behavior of Algorithm~\ref{alg:projected:geodesic}, for $n \in \{ 10, 20, \ldots, 100 \}$ we generate 100 covariance matrices $\covsa$ according to the following procedure. We \emph{(i)} draw a standard normal random matrix $B \in \mathbb{R}^{n \times n}$ and compute $A = B + B^\top$; we \emph{(ii)} conduct an eigenvalue decomposition $A = R D R^T$; we \emph{(iii)} replace $D$ with a random diagonal matrix $\wh D$ whose diagonal elements are sampled uniformly from $[1, 10]^n$; and we \emph{(iv)} set $\covsa = R \wh D R^\top$. For each of these covariance matrices, we set $\hat{\mu} = 0$, $M = 1$, $x_1^M \Let x$ for a standard normal random vector $x \in \mathbb{R}^n$ and calculate the optimistic likelihood~\eqref{eq:FR} for $\rho = \sqrt{n} / 100$. This choice of $\rho$ ensures that the radius of the ambiguity set scales with $n$ in the same way as the Frobenius norm. 
Figures~\ref{fig:algorithm:a} and \ref{fig:algorithm:b} report the number of iterations as well as the overall execution time of Algorithm~\ref{alg:projected:geodesic} when we terminate the algorithm as soon as the relative improvement $ |[L(\cov_{k+1}) - L(\cov_{k})] / L(\cov_{k+1}) | $ drops below $0.01\%$. Notice that the number of required iterations scales linearly with $n$ while the overall runtime grows polynomially with $n$. Figure~\ref{fig:algorithm:c} shows the relative improvement as a function of the iteration count. Empirically, the number of iterations scales with $\mc{O}(1 / k^2)$, which is faster than the theoretical rate established in Theorem~\ref{thm:convergence}. We also study the empirical convergence behavior of Algorithm~\ref{alg:projected:geodesic} when the input matrix $S$ is positive definite. We repeat the first experiment with $M=100$, and we set $S = \delta I + \sum_{i=1}^M x_i x_i^\top / M $ for $\delta = 10^{-6}$ to ensure that $S$ is positive definite. Figure~\ref{fig:algorithm:d} indicates that, in this case, the empirical convergence rate of Algorithm~\ref{alg:projected:geodesic} is linear.

\begin{figure*}
	\centering
	\subfigure[Scaling of iteration count for $S \succeq 0$]{\label{fig:algorithm:a} \includegraphics[width=0.45\columnwidth]{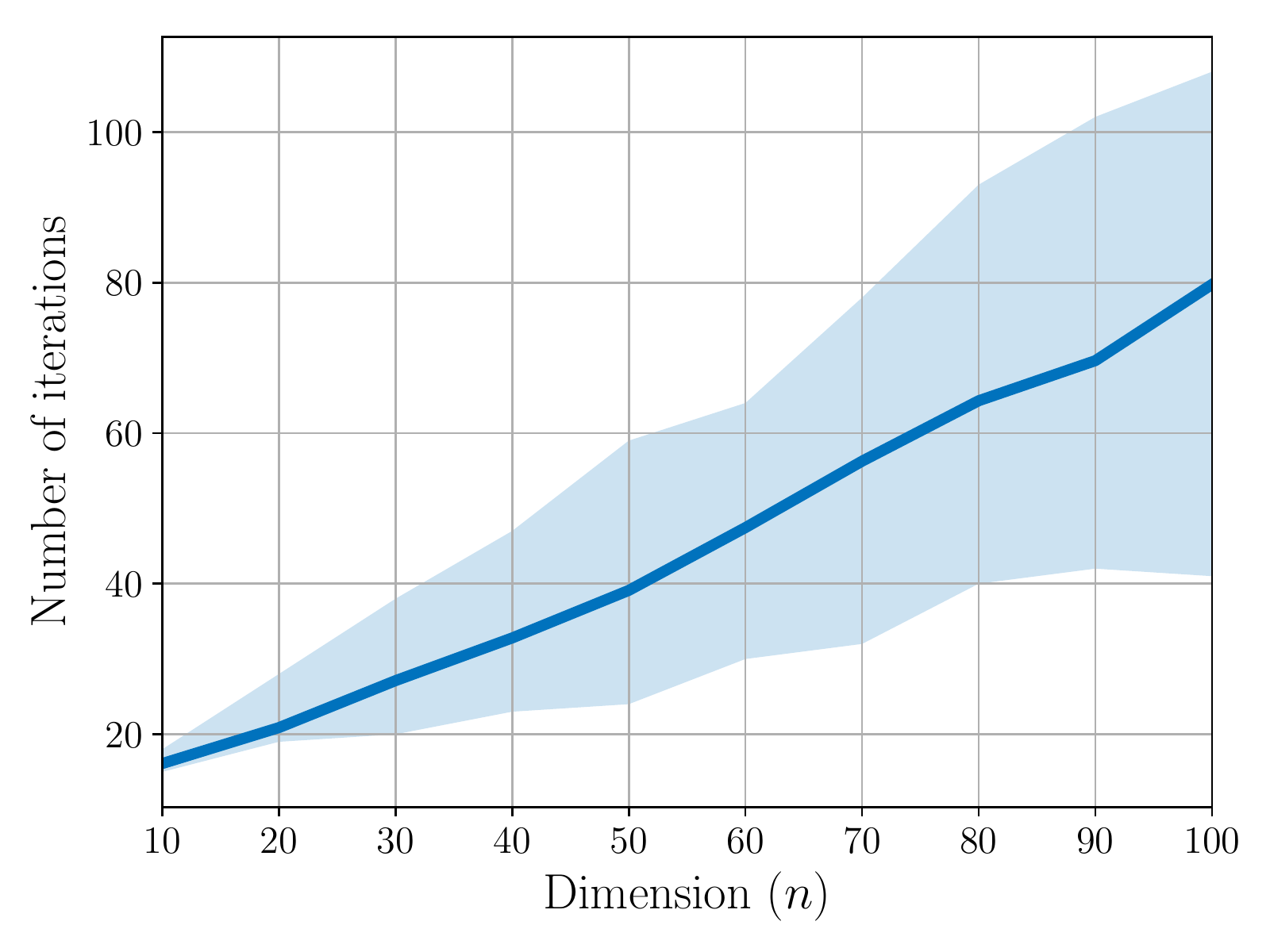}} \hspace{0pt}
	\subfigure[Scaling of execution time  for $S \succeq 0$]{\label{fig:algorithm:b} \includegraphics[width=0.45\columnwidth]{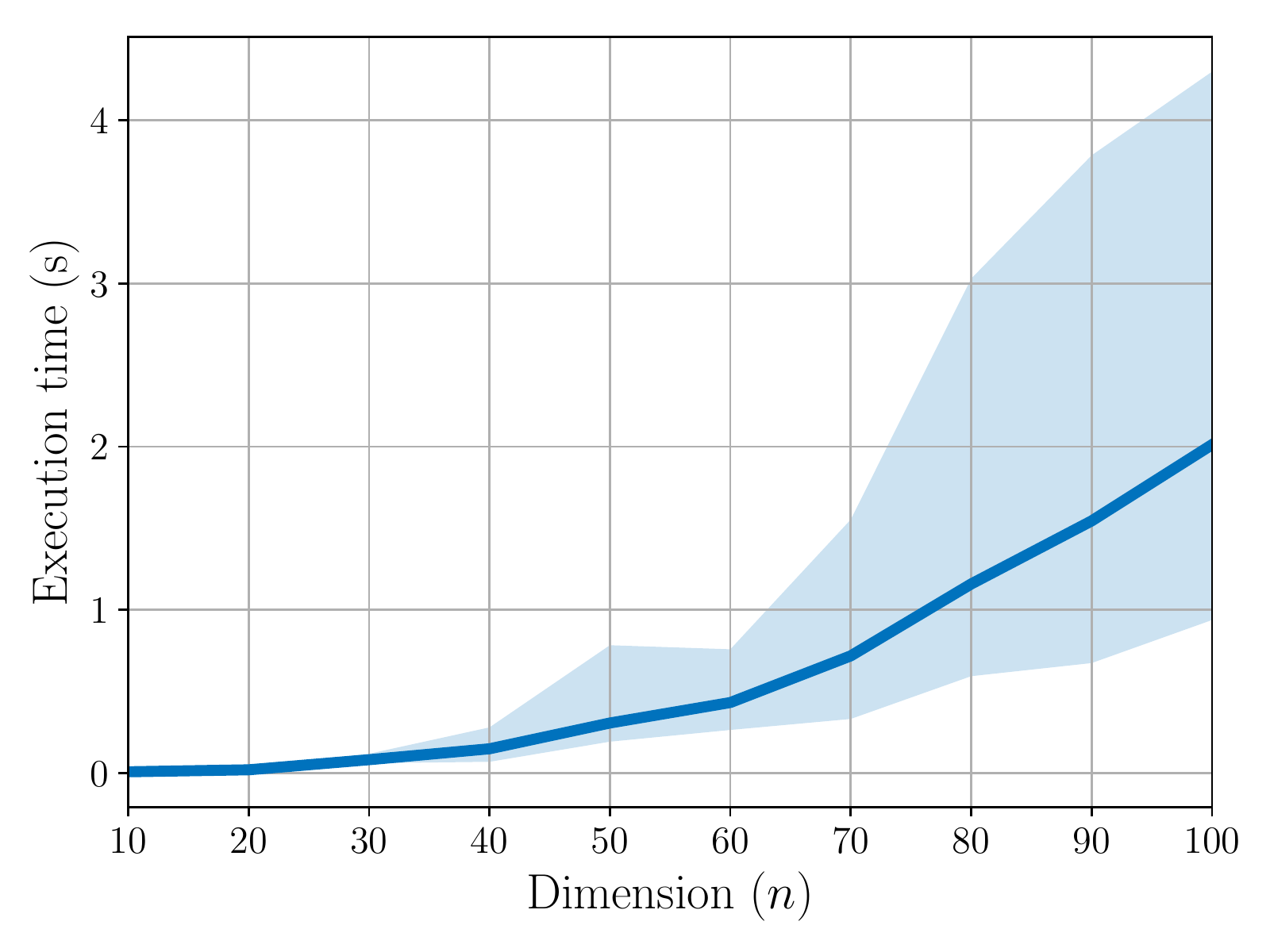}} \hspace{0pt} \\
	\subfigure[Convergence for $n=100$ for $S \succeq 0$]{\label{fig:algorithm:c} \includegraphics[width=0.45\columnwidth]{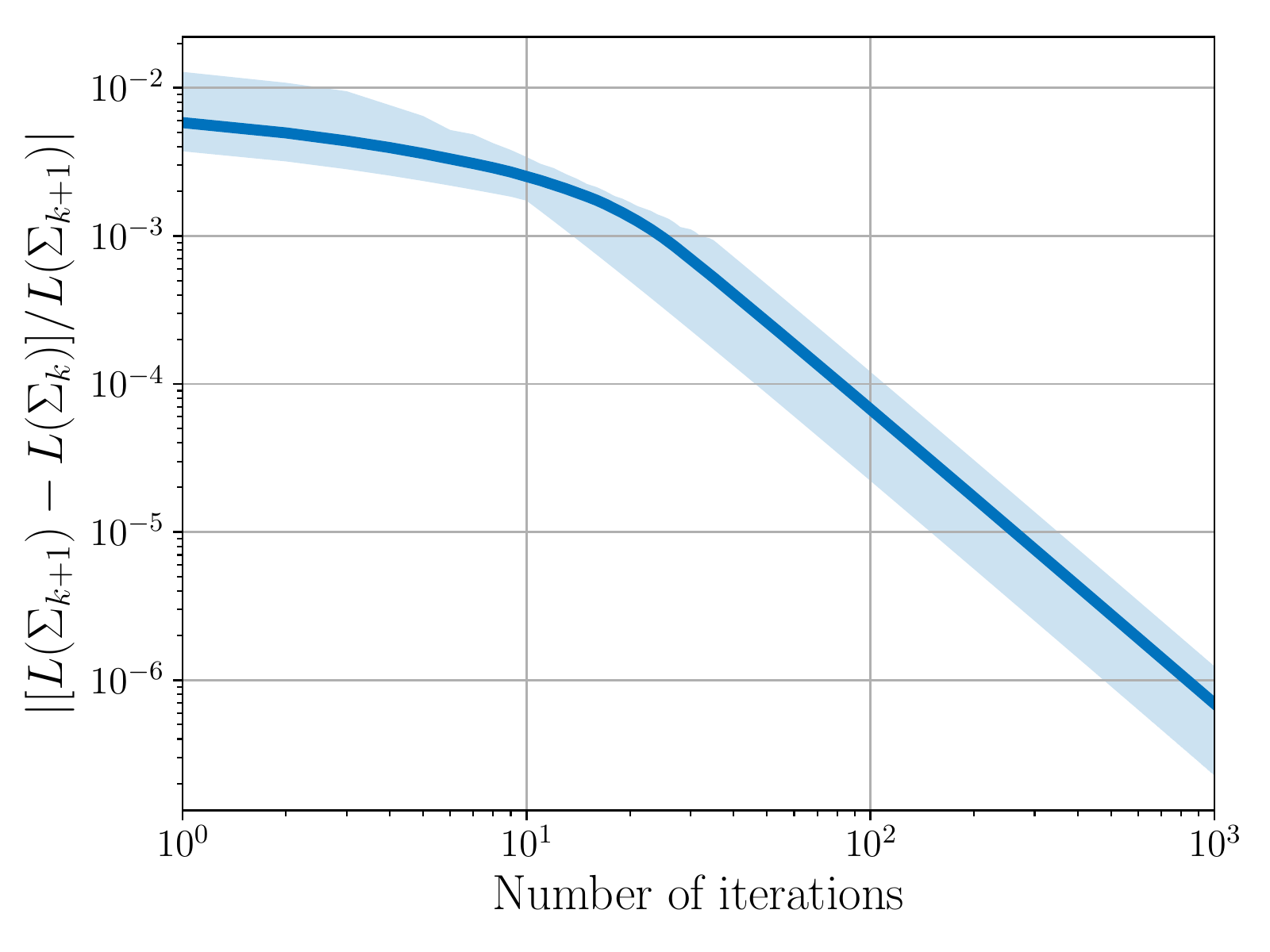}} \hspace{0pt}
	\subfigure[Convergence for $n=100$ for $S \succ 0$]{\label{fig:algorithm:d} \includegraphics[width=0.45\columnwidth]{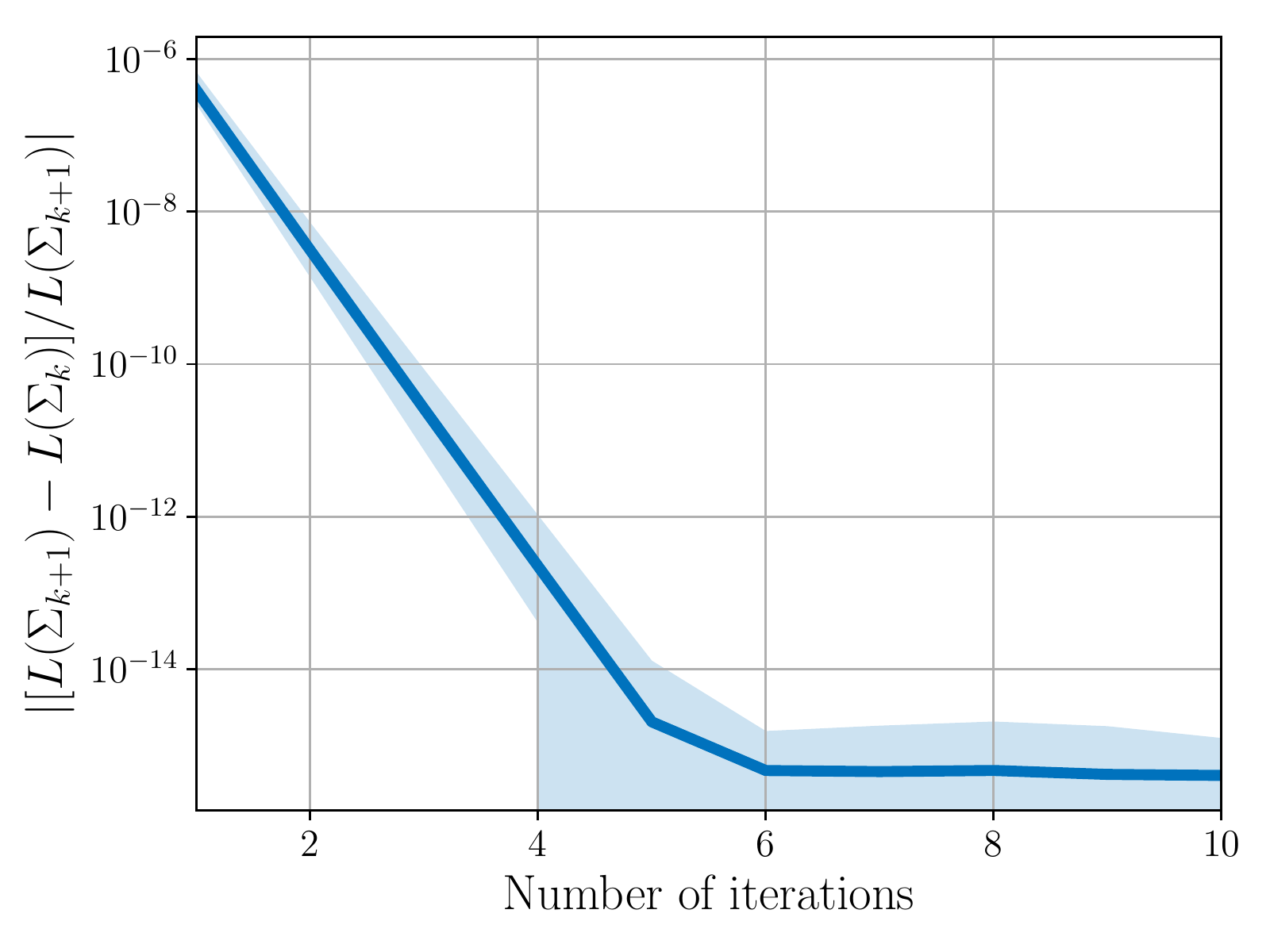}} \hspace{0pt}
	\captionsetup{skip=0pt}
	\caption{Convergence behavior of the projected geodesic gradient descent algorithm. Solid lines (shaded regions) represent averages (ranges) across 100 independent simulations.}
	\label{fig:algorithm}
\end{figure*}

\subsection{Application: Flexible Discriminant Rules}\label{sec:classification}
Consider a classification problem where a categorical response $Y \in \mathcal{C}$, $\mathcal{C} = \{ 1, \ldots, C \}$, should be predicted from continuous inputs $X \in \mathbb{R}^n$. In this context, Bayes' Theorem implies that $\mathbb{P} (Y = c | X = x) \propto \pi_c \cdot f_c (x)$, $c \in \mathcal{C}$, where $\pi_c = \mathbb{P} (Y = c)$ denotes the prior probability of the response belonging to class $c$, and $f_c$ is the density function of $X$ for an observation of class $c$. In practice, $\pi_c$ and $f_c$ are unknown and need to be estimated from a training data set $(\wh x_1, \wh y_1), \ldots, (\wh x_N, \wh y_N)$. Assuming that the densities $f_c$, $c \in \mathcal{C}$, correspond to Gaussian distributions with (unknown) class-specific means $\mu_c$ and covariance matrices $\Sigma_c$, the quadratic discriminant analysis (QDA) replaces $\pi_c$ with $\hat{\pi}_c = N_c / N$, where $N_c = | \{ i : \wh y_i = c \} |$, and $f_c$ with the density of the Gaussian distribution $\hat{\mathbb{P}}_c \sim \mathcal{N} (\hat{\mu}_c, \hat{\Sigma}_c)$, whose mean and covariance matrix are estimated from the training data, to classify a new observation $x$ using the discriminant rule
\begin{equation*}
\mathcal{C}_{\text{QDA}}(x) \in \mathop{\arg \max}_{c \in \mathcal{C}}
\left\{ \frac{1}{2} \ell (x, \hat{\mathbb{P}}_c) + \log (\hat{\pi}_c) \right\} \,.
\end{equation*}
Here, the likelihood $\ell (x, \hat{\mathbb{P}}_c)$ is defined as  in~\eqref{eq:classical_MLE} for $M = 1$.
If $\wh \pi_1 = \ldots = \wh \pi_C$, this classification rule reduces to the maximum likelihood discrimant rule~\citep[\S~14]{ref:Hardle-2015}. 

QDA can be sensitive to misspecifications of the empirical moments. To reduce this sensitivity, we replace the nominal Gaussian distributions $\hat{\mathbb{P}}_c$ with the Gaussian distributions $\mathbb{P}_c^\star$ that would have generated the sample $x$ with highest likelihood, among all Gaussian distributions in the vicinity of the nominal distributions $\hat{\mathbb{P}}_c$. This results in a \emph{flexible discriminant rule} of the form
\begin{equation*}
\mathcal{C}_{\text{flex}} (x) \in \mathop{\arg \max}_{c \in \mathcal{C}} \; \max_{\mathbb{P} \in \mathcal{P}_c}
\left\{ \frac{1}{2} \ell (x, \mathbb{P}) + \log (\hat{\pi}_c) \right\},
\end{equation*}
which makes use of the optimistic likelihoods~\eqref{eq:optimistic_MLE}. Here, $\mathcal{P}_c$ is the FR or KL ball centered at the nominal distribution $\hat{\mathbb{P}}_c$. To ensure that $\hat{\Sigma}_c \succ 0$ for all $c \in \mathcal{C}$, we use the Ledoit-Wolf covariance estimator~\cite{ledoit2004well}, which is parameter-free and returns a well-conditioned matrix by minimizing the mean squared error between the estimated and the real covariance matrix.

We compare the performance of our flexible discriminant rules with standard QDA implementations from the literature on datasets from the UCI repository~\citep{UCI2013}. Specifically, we compare the following methods.
\begin{itemize}[leftmargin=*]
	\item \textbf{FQDA} and \textbf{KQDA}: our flexible discriminant rules based on FR (FQDA) and KL (KQDA) ambiguity sets with radii $\rho_c$;
	\item \textbf{QDA}: regular QDA with empirical means and covariance matrices estimated from data;
	\item \textbf{RQDA}: regularized QDA based on the linear shrinkage covariance estimator $\covsa_c + \rho_c I_n$;
	\item \textbf{SQDA}: sparse QDA based on the graphical lasso covariance estimator~\cite{friedman2008sparse} with parameter $\rho_c$;
	\item \textbf{WQDA}: Wasserstein QDA based on the nonlinear shrinkage approach~\cite{nguyen2018distributionally} with parameter $\rho_c$.
\end{itemize}
All results are averaged across $100$ independent trials for $\rho_c \in \{ a \sqrt{n} \cdot 10^b: a \in \{ 1, \hdots, 9 \}, \; b \in \{ -3, -2, -1 \}  \}$. In each trial, we randomly select $75\%$ of the data for training and the remaining $25\%$ for testing. The size of the ambiguity set and the regularization parameter are selected using stratified $5$-fold cross validation. The performance of the classifiers is measured by the {\em correct classification rate} (CCR). The average CCR scores over the $100$ trials are reported in Table~\ref{table}.

\begin{table}[t]
	\centering
	\caption{Average correct classification rates on the benchmark instances}
	\label{table}
	\begin{tabular}{lcccccccc}
		\toprule
		&   FQDA &   KQDA &   QDA  &   RQDA &   SQDA &   WQDA \\
		\midrule
		Australian              &  80.68 &  \bf{83.68} &  80.03 &  79.76 &  80.73 &  79.94 \\
		Banknote authentication &  99.07 &  \bf{99.47} &  98.56 &  98.54 &  98.53 &  98.54 \\
		Climate model           &  94.46 &  \bf{94.55} &  91.78 &  92.72 &  94.42 &  92.78 \\
		Cylinder                &  70.69 &  70.67 &  67.10 &  70.33 &  \bf{70.99} &  70.34 \\
		Diabetic                &  \bf{75.97} &  74.53 &  74.19 &  74.60 &  74.70 &  75.04 \\
		Fourclass               &  \bf{80.13} &  79.97 &  79.32 &  79.32 &  79.32 &  79.33 \\
		German credit           &  74.50 &  74.60 &  71.41 &  \bf{76.18} &  74.99 &  76.31 \\
		Haberman                &  74.87 &  \bf{75.41} &  74.92 &  74.96 &  75.04 &  74.97 \\
		Heart                   &  \bf{84.23} &  83.31 &  81.42 &  82.62 &  84.17 &  82.42 \\
		Housing                 &  88.89 &  \bf{92.90} &  88.54 &  87.01 &  81.69 &  88.31 \\
		Ilpd                    &  57.42 &  \bf{57.83} &  55.18 &  54.97 &  55.45 &  55.15 \\
		Mammographic mass       &  80.66 &  80.85 &  80.37 &  80.88 &  \bf{81.05} &  80.65 \\
		Pima                    &  \bf{75.97} &  74.53 &  74.19 &  74.60 &  74.70 &  75.04 \\
		Ringnorm                &  \bf{98.69} &  98.65 &  98.56 &  98.56 &  98.65 &  98.56 \\
		\bottomrule
	\end{tabular}
\end{table}

\appendix
\renewcommand\thesection{Appendix~\Alph{section}}


\section{Justification for using Ambiguity Sets with Fixed Mean Vector}
\label{sect:appendix:justification}
\renewcommand\thesection{\Alph{section}}
\setcounter{equation}{0}
\renewcommand{\theequation}{A.\arabic{equation}}
\renewcommand{\thefigure}{A.\arabic{figure}}

We provide some empirical evidence to justify the ambiguity sets with a fixed mean vector used in Sections~\ref{sect:FR} and~\ref{sect:KL}. Towards this end, we first fix a matrix $A \in \R^{n \times n}$, where each element is drawn independently from a standard Gaussian distribution, and set $\cov \Let AA^\top$. We then generate $N \in \{20, \ldots, 100\}$ i.i.d.~samples $\wh x_1, \ldots, \wh x_N$ from the Gaussian distribution $\mbb{Q} = \N(0, \cov)$, and compute the empirical mean $\msa_N$ and the empirical covariance matrix $\covsa_N$ as
\[
\msa_N =\frac{1}{N} \sum_{i=1}^N \wh x_i \quad \text{and} \quad \covsa_N = \frac{1}{N-1} \sum_{i=1}^N (\wh x_i - \msa_N)(\wh x_i - \msa_N)^\top.
\]
We now construct two probability distributionss based on $\msa_N$ and $\covsa_N$, that is, we set
\[
\psa_{N, 0} = \N(0, \covsa_N) \quad \text{and} \quad \psa_{N, \cov} = \N(\msa_N, \cov).
\]
Notice that $\psa_{N, 0}$ has the same mean as the unknown probability distribution $\mbb Q$ that generates the data, while $\psa_{N, \cov}$ has the same covariance matrix as $\mbb{Q}$. 
In the following, we define the mean vector estimation error $\delta_{N, \text{mean}}$ and the covariance matrix estimation error $\delta_{N, \text{covariance}}$ as
\[
\delta_{N, \text{mean}} = \varphi(\psa_{N, \cov}, \mbb{Q}) \quad \text{and} \quad \delta_{N, \text{cov}} = \varphi(\psa_{N, 0}, \mbb{Q}), \quad 
\]	
respectively, where $\varphi(\cdot, \cdot)$ is a dissimilarity measure for distributions that can be set either to the Fisher-Rao metric (see Section~\ref{sect:FR}) or to the Kullback-Leibler divergence (see Section~\ref{sect:KL}). 

Figure~\ref{fig:compare} shows the average estimation error for different sample sizes $N$, where the average is taken over 500 independent simulation runs. We observe that the error in estimating the covariance matrix is one order of magnitude higher than the error in estimating the mean vector under both the KL divergence and the FR metric.

\begin{figure*}[h]
	\centering
	\subfigure[KL divergence]{\label{fig:compare:KL} \includegraphics[width=0.48\columnwidth]{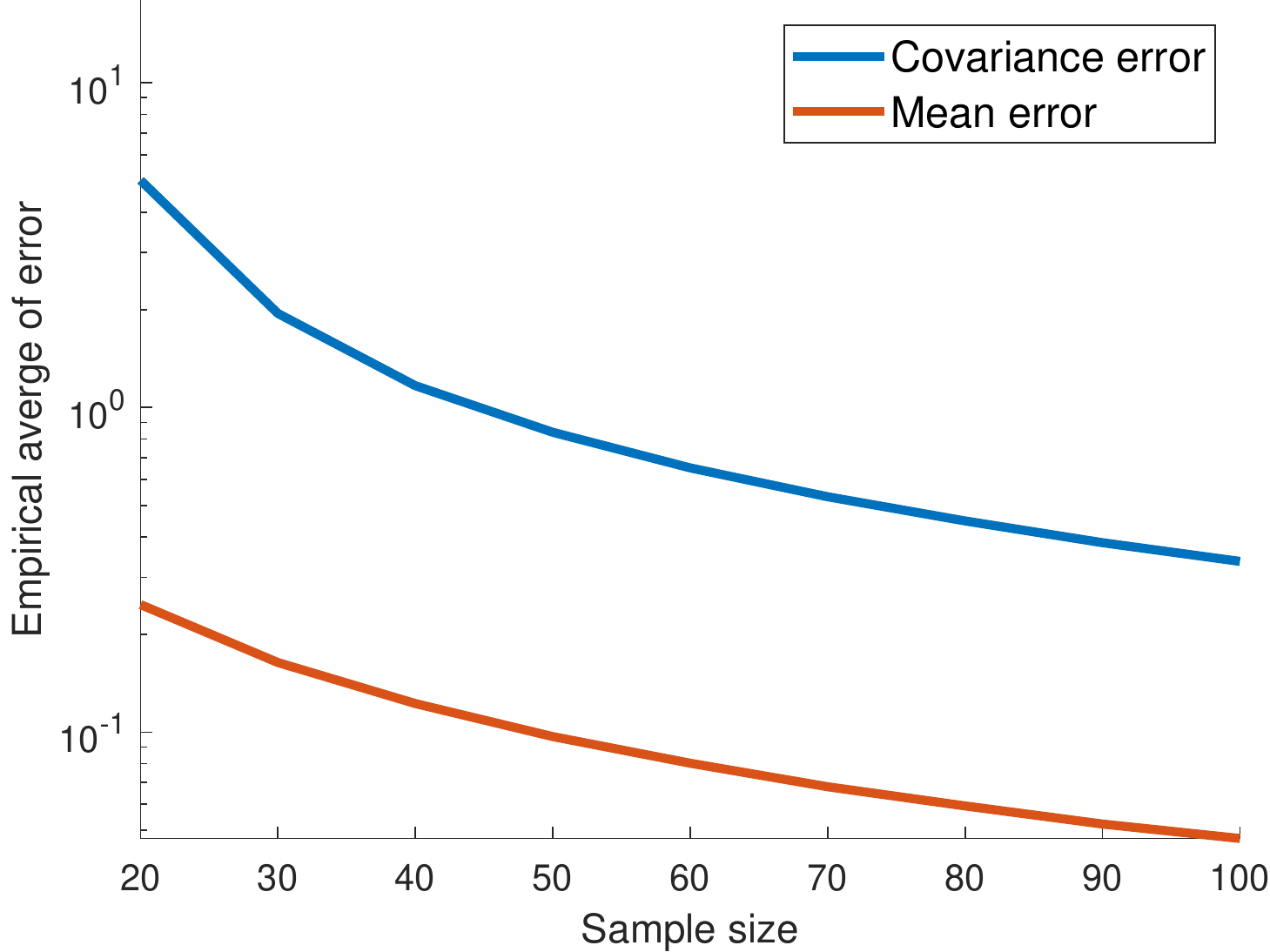}} \hspace{0pt}
	\subfigure[FR metric]{\label{fig:compare:FR} \includegraphics[width=0.48\columnwidth]{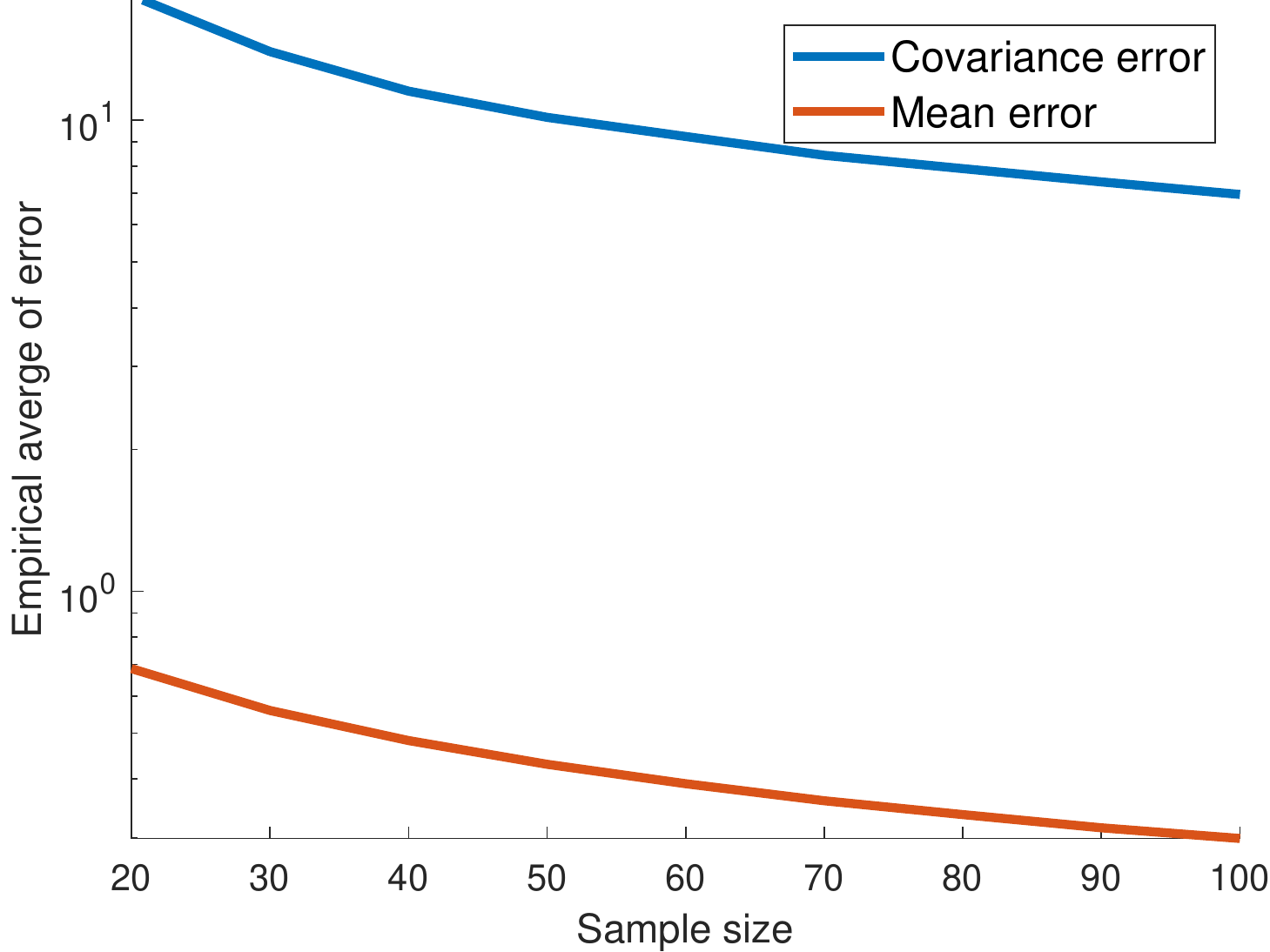}} \hspace{0pt}
	\captionsetup{skip=0pt}
	\caption{Average estimation error for different sample sizes $N$ using the KL divergence or the FR metric as a dissimilarity measure.}
	\label{fig:compare}
\end{figure*}

\renewcommand\thesection{Appendix~\Alph{section}}
\section{Optimistic Likelihood with Ambiguous Mean Vector}
\label{sect:mean}
\renewcommand\thesection{\Alph{section}}
\setcounter{equation}{0}
\renewcommand{\theequation}{B.\arabic{equation}}
We now consider the FR and KL ambiguity sets for the family of Gaussian distributions with a fixed covariance matrix $\covsa \in \PD^n$. We thus consider the manifold $\Theta = \R^n$ of the mean vector $\theta = \m$. The FR distance induced by the FR metric on this manifold is denoted by $\bar{d}(\cdot, \cdot)$ and is again available in closed form. 
\begin{proposition}[FR distance between Gaussian distributions~\cite{atkinson1981rao}]\label{proposition:FR:mean}
	If~$\N(\m_0, \covsa)$ and $ \N(\m_1, \covsa)$ are Gaussian distributions with identical covariance matrix $\covsa \in \PD^n$ and mean vectors $\m_0, \m_1 \in \R^n$, we have
	
	\begin{equation*}
	\bar{d}(\m_0, \m_1) = \sqrt{ (\mu_0 - \mu_1)^\top \covsa^{-1} (\mu_0 - \mu_1) }.
	\end{equation*}
\end{proposition}
Similarly, the KL divergence between two distributions with the same covariance matrix admits a simple analytical expression.
\begin{proposition}[KL divergence between Gaussian distributions]\label{prop:KL:mean}
	For any $\covsa\in \PD^n$ and $\m_0$, $\m_1 \in \R^n$, the KL divergence from $\p_0= \N(\m_0, \covsa)$ to $\p_1 = \N(\m_1, \covsa)$ amounts to
	\begin{equation*}
	\KL(\p_0 \parallel \p_1) = \frac{1}{2} (\m_0 - \m_1)^\top \covsa^{-1} (\m_0 - \m_1) .
	\end{equation*}
\end{proposition}  
Throughout this section, we denote by $\psa = \N(\msa, \covsa)$ and $\p = (\m, \covsa)$ two Gaussian distributions with the same covariance matrix $\covsa \in \PD^n$ but different mean vectors $\msa, \m \in \R^n$, respectively. Propositions~\ref{proposition:FR:mean} and~\ref{prop:KL:mean} imply that the FR distance and the KL divergence of $\psa$ and $\p$ satisfy the relation\footnote{More generally, for arbitrary parametric families of distributions, the second-order Taylor expansion of the KL divergence is given by the FR distance because $
	\KL(\psa \parallel \p) = \frac{1}{2} \FR^2 (\psa, \p) + \mathcal{O}(\FR^3 (\psa, \p)).
	$ See \cite[\S~7.2.2]{kass2011geometrical} for further details.} $2\KL(\psa \parallel \p) = \bar{d}^2 (\msa, \m)$.

With the Fisher-Rao distance as our dissimilarity measure $\varphi$ and given the observations $x_1^M$, the optimistic likelihood problem~\eqref{eq:optimistic_MLE} becomes
\begin{equation} \label{eq:mean:prob}
\begin{aligned}
\Min{\m}  \left\{ \displaystyle \frac{1}{M} \sum_{m=1}^M (x_m - \m)^\top \covsa^{-1} (x_m - \m) + \log \det \covsa : \; (\m - \msa)^\top \covsa^{-1} (\m - \msa) \leq \radius^2 \right\}.
\end{aligned}
\end{equation}

Problem~\eqref{eq:mean:prob} is already a finite convex program but can be further simplified to a univariate convex optimization problem and therefore solved in quasi-closed form.

\begin{theorem}[Optimistic likelihood with mean ambiguity set] \label{thm:mean:prob}
	For any $\msa \in \R^n, \covsa \in \PD^n$ and $\rho > 0$, the optimal value of problem~\eqref{eq:mean:prob} is given by
	\[
	\frac{1}{M} \sum_{m=1}^M (x_m - \m\opt)^\top \covsa^{-1} (x_m - \m\opt) + \log\det \covsa,
	\]
	where $\m\opt = (1+\dualvar\opt)^{-1} \left( \bar{x} + \dualvar\opt \msa \right)$, and $\dualvar\opt$ solves the univariate convex optimization problem
	\begin{equation} \label{eq:KL:mean}
	\Min{\dualvar \geq 0} \; \dualvar \left( \radius^2 - \msa^\top \covsa^{-1} \msa  \right) + \frac{\left( \bar x + \dualvar \msa \right)^\top \covsa^{-1} \left( \bar x + \dualvar \msa \right)}{1+ \dualvar}
	\end{equation}
	with $\bar x = M^{-1} \sum_{m=1}^M x_m$.
\end{theorem}
\begin{proof}
	As $\covsa$ is constant, the minimizers of~\eqref{eq:mean:prob} also solve
	\begin{equation} \label{eq:mean:prob:equiv}
	\begin{aligned}
	\Min{\m}  \left\{ \displaystyle M^{-1} \sum_{m=1}^M (x_m - \m)^\top \covsa^{-1} ( x_m - \m) : \; (\m - \msa)^\top \covsa^{-1} (\m - \msa) \leq \radius^2 \right\}.
	\end{aligned}
	\end{equation}
	Problem~\eqref{eq:mean:prob:equiv} is equivalent to
	\begin{align*}
	& \Min{\m} \Max{\dualvar \geq 0} \left\{ 
	\begin{array}{l}
	\inner{\covsa^{-1}}{ M^{-1} \displaystyle\sum_{m=1}^M (\m -  x_m)(\m - x_m)^\top} + \dualvar \left( \inner{\covsa^{-1}}{(\m - \msa)(\m - \msa)^\top} - \radius^2 \right)
	\end{array}
	\right\} \\
	=& \Max{\dualvar \geq 0} \Min{\m} \left\{ 
	\begin{array}{l}
	-\dualvar \radius^2 + \inner{\covsa^{-1}}{M^{-1}\displaystyle\sum_{m=1}^M (\m -  x_m)(\m -  x_m)^\top + \dualvar (\m - \msa)(\m - \msa)^\top}
	\end{array}
	\right\} ,
	\end{align*}
	where the equality follows from strong duality, which holds because $\radius > 0$ and because $\m = \msa$ constitutes a Slater point for the primal problem~\eqref{eq:mean:prob:equiv}. For any fixed $\dualvar \geq 0$, the inner minimization problem over $\m$ admits the optimal solution
	\[
	\m\opt(\dualvar) = (1+\dualvar)^{-1} \left( \bar{x} + \dualvar \msa \right)
	\]
	with $\bar x = M^{-1} \sum_{m=1}^M x_m$. Thus, the optimal value of~\eqref{eq:mean:prob:equiv} equals
	\begin{equation} \notag
	\Max{\dualvar \geq 0} \; \dualvar \left( \msa^\top \covsa^{-1} \msa  - \radius^2 \right) - \frac{\left( \bar x + \dualvar \msa \right)^\top \covsa^{-1} \left( \bar x + \dualvar \msa \right)}{1+ \dualvar},
	\end{equation}
	which is equivalent to the minimization problem~\eqref{eq:KL:mean}.	By strong duality, given any minimizer $\dualvar\opt$ of problem~\eqref{eq:KL:mean}, an optimal solution for~\eqref{eq:mean:prob:equiv} and also for~\eqref{eq:mean:prob} can be constructed as
	\[
	\m\opt = (1+\dualvar\opt)^{-1} \left( \bar{x} + \dualvar\opt \msa \right).
	\] 
	Substituting $\m\opt$ into the objective function of~\eqref{eq:mean:prob} yields the postulated optimal value. 
\end{proof}

In the following, we provide the first- and second-order derivatives of the objective function of~\eqref{eq:KL:mean}, which can be used for implementing the optimization algorithm to solve for $\dualvar\opt$. To this end, we denote by $g(\dualvar)$ the objective function of~\eqref{eq:KL:mean}. A direct calculation shows that
\begin{align*}
g'(\dualvar) &= \left(\radius^2 - \msa^\top \covsa^{-1} \msa  \right) + \frac{((2+ \dualvar)\msa - \bar{x})^\top \covsa^{-1} (\bar{x}+ \dualvar \msa)}{(1+\dualvar)^2}.
\end{align*}
Moreover, the second-order derivative of $g(\dualvar)$ is given by
\begin{align*}
g''(\dualvar) =  \frac{2\msa^\top \covsa^{-1}\msa}{(1+\dualvar)} - \frac{2[(2+\dualvar)\msa - \bar x]^\top \covsa^{-1} (\bar x + \dualvar \msa)}{(1+\dualvar)^3}.
\end{align*}

\renewcommand\thesection{Appendix~\Alph{section}}
\section{Proofs of Section~\ref{sect:FR}}
\label{sect:proof}
\renewcommand\thesection{\Alph{section}}
\setcounter{equation}{0}
\renewcommand{\theequation}{C.\arabic{equation}}
To prove Lemma~\ref{lemma:opt_exist}, we require the following preparatory lemma.

\begin{lemma}[Properties of $\B^{\FR}$]\label{lemma:compact}
	The FR ball has the following properties:
	\begin{enumerate}[wide = 0pt, labelwidth = 2em, labelsep*=0em, itemindent = 0pt, leftmargin = \dimexpr\labelwidth + \labelsep\relax, noitemsep,topsep = 1ex, font=\normalfont, label=(\roman*)]
		\item \label{lemma:compact_i}$\B^{\FR}$ is compact and complete on $\PD^n$. 
		\item \label{lemma:compact_iii}For any $\cov\in \B^{\FR}$, we have $\eig_{\min}(\covsa) e^{-\sqrt{2}\rho}\cdot  I_n \preceq \cov \preceq \eig_{\max}(\covsa) e^{\sqrt{2}\rho}\cdot I_n$.
	\end{enumerate}
\end{lemma}
\begin{proof}
	To prove assertion~\ref{lemma:compact_i}, we first show that $\B^{\FR}$ is compact and complete with respect to the topology induced by the Riemannian distance $d(\cdot,\cdot)$.	
	Recall that $\mathbb{S}^n_{++}$ is a Hadamard manifold and thus constitutes a complete metric space. By the Hopf-Rinow theorem~\citep[\S~8, Theorem~2.8(b)]{carmo1992riemannian}, $\B^{\FR}$ is compact in the usual topology because $\B^{\FR}$ is  a metric ball and therefore closed and bounded. Moreover, $\B^{\FR}$ is complete in the usual topology because any closed subset of a complete metric space is complete as well.
	By \cite[Theorem~13.29]{lee2013smooth}, the metric topology with respect to $d(\cdot, \cdot)$ on $\mathbb{S}^n_{++}$ coincides with the subspace topology of $\PD^n$ with respect to the usual topology on $\mathbb{S}^n$. This completes the proof of assertion~ \ref{lemma:compact_i}.
	
	To prove assertion~\ref{lemma:compact_iii}, pick any $\cov \in \B^{\FR}$ and let $0\leq \eig_1(A) \leq \ldots \leq \eig_n(A)$ denote the eigenvalues of any symmetric positive definite $n$-by-$n$  matrix $A$ in increasing order. Then, we have
	$$\sqrt{\log^2(\lambda_i(\covsa^{-\half} \cov \covsa^{-\half}))} \leq \sqrt{\sum_{j=1}^n \log^2(\lambda_j(\covsa^{-\half} \cov \covsa^{-\half}))} = \|\log( \covsa^{-\half} \cov \covsa^{-\half} )\|_F \leq \sqrt{2}\rho$$
	for any $i = 1, \ldots, n$, where the equality follows from the definition of the Frobenius norm, and the last inequality follows from the definition of $\B^{\FR}$. Note that $ \lambda_i(\covsa^{-\half} \cov \covsa^{-\half}) = 1 / \lambda_{n- i + 1}(\cov^{-\half} \covsa \cov^{-\half})$, 
	and hence any eigenvalue $\eig_i(\covsa^{-\half} \cov \covsa^{-\half})$ obeys the bounds
	$$e^{-\sqrt{2}\rho} \leq \eig_i(\covsa^{-\half} \cov \covsa^{-\half}) \leq e^{\sqrt{2}\rho}.$$
	This implies that
	\[
	\eig_{\max}^{-1}(\covsa) \eig_{\max}(\cov) \leq e^{\sqrt{2}\rho} \quad \text{and} \quad
	\eig_{\min}^{-1}(\covsa) \eig_{\min}(\cov) \geq e^{- \sqrt{2}\rho},
	\]
	which completes the proof of assertion~\ref{lemma:compact_iii}.	
\end{proof}

We are now ready to prove Lemma~\ref{lemma:opt_exist}.

\begin{proof}[Proof of Lemma~\ref{lemma:opt_exist}]
	First, assertion~\ref{lemma:compact_i} of Lemma~\ref{lemma:compact} ensures that the feasible region $\B^{\FR}$ is compact. Second, we note that the objective function $L(\cdot)$ is continuous at any positive definite matrix. By assertion~\ref{lemma:compact_iii} of Lemma~\ref{lemma:compact}, there is a uniform positive lower bound on the eigenvalues of all matrices in $\B^{\FR}$. Therefore $L(\cdot)$ is continuous on $\B^{\FR}$. The solvability of problem~\eqref{eq:FR} then follows from Weierstrass' extreme value theorem~\cite[Corollary~2.35]{aliprantis06hitchhiker}. 
\end{proof}

\begin{proof}[Proof of Theorem~\ref{thm:convex}]
	We first show that $\B^{\FR}$ is a geodesically convex set. By~\cite[Proposition~II.1.4]{ref:Bridson-1999}, balls in CAT($\kappa$) spaces\footnote{A formal definition of CAT spaces can be found in~\citep[Definition~II.1.1]{ref:Bridson-1999}, and the upper bound $\kappa$ of the curvature of a metric space is defined in~\citep[Definition~II.1.2]{ref:Bridson-1999}} of radius less than $D_\kappa/2$ are geodesically convex, where $D_\kappa$ is the diameter of the model space of constant curvature $\kappa$ (see \citep[Definition~I.2.10]{ref:Bridson-1999}). It is known that the smooth manifold $\mathbb{S}^n_{++}$ is a CAT($0$) space \citep[Theorem~II.10.39]{ref:Bridson-1999}, which implies via  \citep[Point~I.2.12]{ref:Bridson-1999} that $D_0 =\infty$. The claim thus follows.
	
	The proof that $L(\cdot)$ is a geodesically convex function over $\PD^n$ closely follows from~\cite[Lemma~III.2]{ref:Zhang-2012} and~\cite[Corollary~5.3]{ref:Sra-2015} and is thus omitted. 
\end{proof}

\begin{proof}[Proof of Lemma~\ref{lemma:projection}]
	The claim trivially holds if $\rho'\le \rho$. We thus prove the two statements under the assumption that $\rho' > \rho$.
	By \citep[Theorem~II.10.39]{ref:Bridson-1999}, $\mathbb{S}^n_{++}$ is a CAT($0$) space. Furthermore, by Lemma~\ref{lemma:compact}, the geodesic ball $\mathcal{B}^{\FR}$ is both complete and compact. Assertion~\ref{item:1} in Lemma~\ref{lemma:projection} then follows from~\cite[Proposition~II.2.4]{ref:Bridson-1999}.
	
	To prove assertion~\ref{item:2}, we define
	$$\displaystyle \cov_p \Let \covsa^\half ( \covsa^{-\half} \cov' \covsa^{-\half} )^{\frac{\rho}{\rho'}} \covsa^\half.$$ One readily verifies that $ d(\covsa, \cov_p) = \rho$, and hence $\cov_p \in \B^{\FR}$. Recall that $\cov' \in \PD^n$ and $d(\covsa, \cov') = \rho'$.
	Given any $\cov \in \B^{\FR}$, by the triangle inequality, we thus have 
	$$ d(\cov, \cov') \geq d(\covsa, \cov') - d(\covsa, \cov) \geq d(\covsa, \cov') - \max_{\cov'' \in \B^{\FR}} d(\covsa, \cov'') = \rho' - \rho.$$ This reasoning implies that 
	\begin{equation}\label{ineq:1}
	\min_{\cov\in \B^{\FR}} d(\cov, \cov') \ge \rho' - \rho.
	\end{equation}
	By definition, the geodesic $\gamma(t) = \covsa^\half ( \covsa^{-\half} \cov' \covsa^{-\half} )^{t} \covsa^\half $ connecting $\covsa$ and $\cov'$ has constant-speed, that is, $d(\gamma(t), \gamma(s)) = d(\gamma(0),\gamma(1)) \cdot | t- s|$ for any $t,s\in [0,1]$ (see~\cite[Theorem 6.1.6]{bhatia2009positive}). Therefore, we have
	$$ d(\cov_p, \cov') = d(\gamma(\tfrac{\rho}{\rho'}), \gamma(1)) = d(\covsa, \cov') \cdot \left|\frac{ \rho}{\rho'} - 1\right| = \rho' - \rho,$$
	which implies that the lower bound~\eqref{ineq:1} is attained by $\cov_p$. The uniqueness result of assertion~\ref{item:1} thus allows us to conclude that $\cov_p$ is the projection of $\cov'$ onto $\B^{\FR}$.
\end{proof}

The proof of Theorem~\ref{thm:convergence} is based on the following two technical lemmas.

\begin{lemma}[Bounded gradient] \label{lemma:bounded:gradient}
	For any $X\in T_\cov \mathbb{S}^n_{++}$, denote by $\|X\|_\cov \Let \sqrt{ \langle X, X \rangle_\cov }$ the norm induced by the inner product $\langle \cdot,\cdot\rangle_\cov$ defined in \eqref{eq:metric}. The Riemannian gradient of the objective function $L(\cdot)$ of problem~\eqref{eq:FR} satisfies
	$$\| \grad L (\cov)\|_{\cov} \leq \sqrt{n}\cdot e^{2\sqrt{2}\rho} \cdot\eig_{\min}^{-2}(\covsa) \cdot \max\{ |1 - e^{\sqrt{2}\rho} \eig^{-1}_{\min}(\hat{\Sigma}) \eig_{\max} (S)| , 1 \} \quad  \forall \cov \in \B^{\FR}.$$
\end{lemma}

\begin{proof}
	By \eqref{eq:grad} and the definition of $\| \cdot \|_\cov$, we have
	\begin{align*}
	\| \text{grad}L (\cov) \|_\cov^2 &= \frac{1}{2}\Tr{\text{grad}L (\cov) \cdot \cov^{-1}\cdot \text{grad}L (\cov)\cdot \cov^{-1}} =\frac{1}{2} \Tr{A \cov^{-2} A \cov^{-2}},
	\end{align*}
	where $A \Let (I_n - \cov^{-\half} S \cov^{-\half}) $. Lemma~\ref{lemma:compact}\ref{lemma:compact_iii} thus implies that
	\[
	(1 - e^{\sqrt{2}\rho}\eig_{\min}^{-1}(\covsa) \eig_{\max}(S) ) I_n \preceq (1 - \eig_{\min}^{-1}(\cov) \eig_{\max}(S) ) I_n \preceq A \preceq I_n,
	\]
	and therefore we have
	\begin{align*}
	\| \text{grad}L (\cov) \|_\cov &\leq \sqrt{ \frac{n}{2} \cdot\frac{ \eig^2_{\max}(A)}{\eig_{\min}^{4}(\cov)} } \\
	&\leq \sqrt{ \frac{n}{2} \cdot\frac{ \max\{1, (1 - e^{\sqrt{2}\rho}\eig_{\min}^{-1}(\covsa) \eig_{\max}(S))^2 \}}{\eig_{\min}^{4}(\covsa) e^{-4\sqrt{2}\rho}} } \\
	&= \frac{\sqrt{n}\cdot \max\left\{1, \left|1 - e^{\sqrt{2}\rho}\eig_{\min}^{-1}(\covsa) \eig_{\max}(S) \right| \right\} }{\sqrt{2}\eig_{\min}^{2}(\covsa) e^{-2\sqrt{2}\rho}},
	\end{align*}
	where the last inequality follows from Lemma~\ref{lemma:compact}\ref{lemma:compact_iii}. This observation completes the proof.
\end{proof}

\begin{lemma}[Lower bounded sectional curvature]\label{lemma:curv_bound}
	The sectional curvature of the Riemannian manifold $\mathbb{S}^n_{++}$ equipped with the FR metric \eqref{eq:metric} is lower bounded by $- 2$.
\end{lemma}
\begin{proof}
	Select $\cov\in \mathbb{S}^n_{++}$, and let $X,Y\in T_{\cov}\mathbb{S}^n_{++} $ be two orthonormal tangent vectors at $\cov$, that is, $$ \|X\|_\cov = 1 = \|Y\|_\cov\quad\text{and}\quad \langle X , Y \rangle_\cov = 0.$$ Using the formula for the Riemannian curvature tensor $R(\cdot,\cdot,\cdot,\cdot)$ from \citep[Theorem 2.1 (ii)]{skovgaard1984riemannian}, we have
	\begin{equation}\label{eq:1}
	R(X,Y,Y,X) = -\frac{1}{4} \Tr{ Y\cov^{-1}X\cov^{-1}X\cov^{-1}Y\cov^{-1} } + \frac{1}{4}\Tr{ X\cov^{-1}Y\cov^{-1}X\cov^{-1}Y\cov^{-1} }.
	\end{equation}
	Then, the sectional curvature $\kappa(X,Y)$ associated with the 2-plane spanned by $\{X,Y\}$ satisfies
	\begin{align*}
	\kappa(X,Y) =& - R(X,Y,X,Y) \\
	=&\, -\frac{1}{4} \Tr{ Y\cov^{-1}X\cov^{-1}X\cov^{-1}Y\cov^{-1} } + \frac{1}{4}\Tr{ X\cov^{-1}Y\cov^{-1}X\cov^{-1}Y\cov^{-1} } \\
	\ge&\, -\left( \|X\|^2_\cov \|Y\|^2_\cov + \|X\|^2_\cov \|Y\|^2_\cov\right)= - 2,
	\end{align*}
	where the first equality follows from \cite[Proposition 8.8]{lee1997riemannian}, the second equality exploits~\eqref{eq:1}, and the inequality holds due to the Cauchy-Schwarz inequality.
\end{proof}

We are now equipped with all the necessary ingredients to prove Theorem~\ref{thm:convergence}.

\begin{proof}[Proof of Theorem~\ref{thm:convergence}]
	The proof closely follows that of~\cite[Theorem 9]{ref:Zhang-2016}. The main difference is that we replace the assumption of Lipschitz continuity of the objective function with the assumption of a bounded Riemannian gradient. Due to Theorem~\ref{thm:convex}, the function $L(\cdot)$ is geodesically convex. So we have (see the sentence following Definition 2 in \cite{ref:Zhang-2016})
	\begin{equation*}
	L(\cov') \ge L(\cov) + \langle \text{grad} L (\cov), \Exp_{\cov}^{-1} (\cov') \rangle_{\cov}, \quad \forall \cov,\cov'\in \mathbb{S}^n_{++}.
	\end{equation*}
	Therefore, for any $k\ge 1$, 
	\begin{equation}\label{ineq:2}
	L(\cov_k ) - L(\cov^\star) \le - \langle \text{grad} L (\cov_k), \Exp_{\cov_k}^{-1} (\cov^\star) \rangle_{\cov_k}.
	\end{equation}
	By \cite[Corollary 8]{ref:Zhang-2016}, Lemma~\ref{lemma:curv_bound} and because the diameter of the feasible region is $2\rho$, the right hand side of~\eqref{ineq:2} is upper bounded by
	\begin{equation}\label{ineq:3}
	\frac{1}{2\alpha} \left( d^2(\cov_k, \cov^\star) - d^2(\cov_{k+1}, \cov^\star)\right) + \frac{\alpha\cdot  (2\rho)\cdot\sqrt{2} \cdot \| \text{grad} L (\cov_k) \|_{\cov_k}^2}{2\tanh((2\rho)\cdot\sqrt{2})} ,
	\end{equation}
	where the norm $\|\cdot\|_{\cov_k}$ is defined as in Lemma~\ref{lemma:bounded:gradient}.
	Therefore, substituting the upper bound \eqref{ineq:3} into \eqref{ineq:2} and using \ref{lemma:bounded:gradient}, we find
	\begin{equation}\label{ineq:4}
	L(\cov_k ) - L(\cov^\star) \le \frac{1}{2\alpha} \left( d^2(\cov_k, \cov^\star) - d^2(\cov_{k+1}, \cov^\star)\right) + \frac{ \sqrt{2}\alpha\rho \Gamma^2}{\tanh(2\sqrt{2}\rho)},
	\end{equation}
	where $\Gamma \Let 2^{-1/2}\sqrt{n}\cdot e^{2\sqrt{2}\rho} \cdot\eig_{\min}^{-2}(\covsa) \cdot \max\{ |1 - e^{\sqrt{2}\rho} \eig^{-1}_{\min}(\hat{\Sigma}) \eig_{\max} (S)| , 1 \}$.
	By telescoping, we then obtain
	\begin{align*}
	\frac{1}{K} \sum_{k=1}^{K} L(\cov_k ) - L(\cov^\star) & \le \frac{1}{2\alpha K} \left( d^2(\cov_1, \cov^\star) - d^2(\cov_{K+1}, \cov^\star)\right) + \frac{ \sqrt{2}\alpha\rho \Gamma^2}{\tanh(2\sqrt{2}\rho)} \\
	& \le \frac{2\rho^2}{\alpha K}  + \frac{ \sqrt{2}\alpha\rho \Gamma^2}{\tanh(2\sqrt{2}\rho)} = \frac{2^{\frac{7}{4}} \rho^{\frac{3}{2} }\Gamma}{\sqrt{K \tanh(2\sqrt{2}\rho)}},
	\end{align*}
	where the second inequality follows from the bounds $d^2(\cov_{K+1}, \cov^\star)\ge 0 $ and $d(\cov_1, \cov^\star) \le 2\rho$, and the equality holds because $\alpha = 2^{1/4} \sqrt{\rho \tanh(2\sqrt{2}\rho)}/(\Gamma \sqrt{K})$. Note that although the matrix $\cov_{K+1}$ is not actually computed because Algorithm~\ref{alg:projected:geodesic} is terminated at $k = K-1$, it is well-defined, and inequality~\eqref{ineq:4} is valid for $k = K$.
	If we can show that $$L(\bar{\cov}_K) \le \frac{1}{K}\sum_{k=1}^{K} L( \cov_k ),$$ the desired result follows. Towards that end, we prove by induction that
	\begin{equation}\label{ineq:5}
	L(\bar{\cov}_T) \le \frac{1}{T}\sum_{k=1}^{T} L( \cov_k ) \quad \forall T \in \mathbb{N}t.
	\end{equation}
	The inequality trivially holds for $T = 1$. Suppose now that the inequality in~\eqref{ineq:5} holds for some $T\ge 1$. Then, we have 
	\begin{align*}
	L(\bar{\cov}_{T+1}) & = L\left( \bar{\cov}_T^{\tfrac{1}{2}} \left( \bar{\cov}_T^{-\tfrac{1}{2}}\cov_{T+1}\bar{\cov}_T^{-\tfrac{1}{2}} \right)^{\tfrac{1}{T+1}} \bar{\cov}_T^{\tfrac{1}{2}} \right) \\
	& = L\left( \gamma_{T} \left( \frac{1}{T+1}\right) \right) \\
	& \le \frac{T}{T+1} L\left( \bar{\cov}_T \right) + \frac{1}{T+1} L\left( \cov_{T+1} \right) \\
	& \le \frac{1}{T+1} \sum_{k=1}^{T} L( \cov_k ) + \frac{1}{T+1} L\left( \cov_{T+1} \right) \\
	& = \frac{1}{T+1} \sum_{k=1}^{T+1} L( \cov_k ),
	\end{align*}
	where $\gamma_T$ denotes the geodesic from $\bar{\cov}_T$ to $\cov_{T+1}$. The first inequality follows from the geodesic convexity of $L(\cdot)$ (see Theorem~\ref{thm:convex} and Definition~\ref{def:g-convex}), and the second inequality holds due to the induction hypothesis~\eqref{ineq:5}. The claim now follows because $T$ was chosen arbitrarily.
\end{proof}

Next,  we formally define the notions of geodesic strong convexity and geodesic smoothness for functions on $\PD^n$.

\begin{definition}[Strong convexity] \label{def:strong-convexity}
	Let $\mathcal{B} \subseteq \mathbb{S}^n_{++}$ be a subset and $\sigma>0$. A differentiable function $F : \mathcal{B} \rightarrow \mathbb{R}$ is said to be (geodesically) $\sigma$-strongly convex on $\mathcal{B}$ if 
	\begin{equation}
	F(Y) \ge F(X) + \langle \text{grad}\, F(X) , \Exp_X^{-1}(Y) \rangle_X + \frac{\sigma}{ 2 } d^2(X,Y).
	\end{equation}
\end{definition}
\begin{definition}[Smoothness] \label{def:smoothness}
	Let $\mathcal{B} \subseteq \PD^n$ be a subset and $\beta>0$. A differentiable function $F : \mathcal{B} \rightarrow \mathbb{R}$ is said to be (geodesically) $\beta$-smooth on $\mathcal{B}$ if 
	\begin{equation}
	F(Y) \le F(X) + \langle \text{grad}\, F(X) , \Exp_X^{-1}(Y) \rangle_X + \frac{\beta}{ 2 } d^2(X,Y).
	\end{equation}
\end{definition}
The proof of Lemma~\ref{lemma:obj_function} is based on the following preparatory results.
\begin{lemma}\label{lem:F}
	Let $F:\mathbb{S}^n_{++} \rightarrow \mathbb{R}$ be a twice continuously differentiable function and $\mathcal{B}\subseteq \PD^n$ be a geodesically convex subset. The following implications hold.
	\begin{enumerate}[label=(\roman*),leftmargin=*]
		\item\label{lem:F_sc} If the smallest eigenvalue of the Riemannian Hessian $\mathrm{hess}\, F(X)$ (interpreted as an operator on $T_X \mathbb{S}^N_{++}$) of $F$ at $X$ is lower bounded uniformly on $\mathcal{B}$ by $\sigma>0$, \ie,
		\begin{equation}
		\min\left\lbrace \langle \mathrm{hess}\,F(X)[V], V \rangle : V\in T_X \mathbb{S}^N_{++}, \|V\|_X = 1 \right\rbrace \ge \sigma \quad \forall X\in \mathcal{B},
		\end{equation}
		then $F$ is $\sigma$-strongly convex on $\mathcal{B}$.
		\item\label{lem:F_s} If the largest eigenvalue of the Riemannian Hessian $\mathrm{hess}\, F(X)$ (interpreted as an operator on $T_X \mathbb{S}^N_{++}$) of $F$ at $X$ is upper bounded uniformly on $\mathcal{B}$ by $\beta >0$, \ie,
		\begin{equation}
		\max\left\lbrace \langle \mathrm{hess}\,F(X)[V], V \rangle : V\in T_X \mathbb{S}^N_{++}, \|V\|_X = 1 \right\rbrace \le \beta \quad \forall X\in \mathcal{B},
		\end{equation}
		then $F$ is $\beta$-smooth on $\mathcal{B}$.
	\end{enumerate}
\end{lemma}
The proof of Lemma~\ref{lem:F} closely follows that of its Euclidean counterpart and is omitted here.
\begin{proof}[Proof of Lemma~\ref{lemma:obj_function}]
	Define $f(\Sigma) = \Tr{ \Sigma^{-1} S } $. Because $\log\det\cov$ is a geodesically linear function~\cite[Proposition 12]{sra2018geodesically}, it suffices to study the smoothness and convexity properties of $f(\cdot)$. By \cite[Equations~(28)]{ferreira2018gradient}, the Riemannian Hessian $\mathrm{hess} f(\Sigma)$ at $\Sigma$ is given by
	\begin{equation}\label{eq:R_Hess}
	\mathrm{hess} f(\Sigma) [V] =  \Sigma \left( \nabla^2 f(\Sigma) [V] \right) \Sigma + \frac{1}{2} \left( V \nabla f(\Sigma) \Sigma + \Sigma \nabla f(\Sigma) V \right) \quad\forall V\in T_\Sigma \mathbb{S}^n_{++} .
	\end{equation}	
	By elementary matrix calculus, we know that 
	\begin{equation}\label{eq:f_grad}
	\nabla f (\Sigma) = - \Sigma^{-1} S \Sigma^{-1}
	\end{equation}
	and 
	\begin{equation}\label{eq:f_hess}
	\nabla^2 f (\Sigma) [V] = \Sigma^{-1} V \Sigma^{-1} S \Sigma^{-1} + \Sigma^{-1} S \Sigma^{-1} V \Sigma^{-1} \quad\forall V\in  \mathbb{S}^n,
	\end{equation}
	where the Hessian $\nabla^2 f (\Sigma)$ is interpreted as a linear operator on $\mathbb{S}^n$.
	Noting that $T_\Sigma \mathbb{S}^n_{++} = \mathbb{S}^n $ and combining \eqref{eq:R_Hess}, \eqref{eq:f_grad} and \eqref{eq:f_hess}, we obtain
	\begin{equation}\label{eq:hess_eig}
	\langle \mathrm{hess} f(\Sigma) [V], V \rangle_{\Sigma} =  \Tr{  \Sigma^{-1} S \Sigma^{-1} V \Sigma^{-1} V  } \quad\forall V\in  \mathbb{S}^n.
	\end{equation}
	Using these preparatory results, we now demonstrate that $f(\cdot)$ is $\beta$-smooth and $\sigma$-strongly convex in the geodesic sense.
	
	\textbf{Smoothness.} In order to establish the smoothness properties of $f(\cdot)$, we consider the maximization problem
	\begin{equation*}
	\max \left\{ \langle \mathrm{hess} f(\Sigma) [V], V \rangle_{\Sigma} : V \in \mbb S^n,~\|V\|_\Sigma = 1 \right\},
	\end{equation*}
	which, by \eqref{eq:hess_eig} and the definition of $\|\cdot\|_\Sigma$, is equivalent to
	\[
	\max \left\{ \Tr{ \Sigma^{-1} S \Sigma^{-1} V \Sigma^{-1} V } : V \in \mbb S^n,~\tfrac{1}{2}\Tr{\Sigma^{-1} V \Sigma^{-1} V } = 1 \right\}.
	\]		
	The optimal value of this problem is upper bounded by
	$2\,\lambda_{\max}(S) / \lambda_{\min}(\Sigma)$. Using the bound from Lemma~\ref{lemma:compact}\ref{lemma:compact_iii}, we have 
	\[
	\frac{2\,\lambda_{\max}(S)}{\lambda_{\min}(\Sigma)} \le \frac{2 \lambda_{\max}(S)}{\lambda_{\min}(\widehat{\Sigma}) \exp(-\sqrt{2}\rho)} = \beta.
	\]
	By Lemma~\ref{lem:F}\ref{lem:F_s}, $f(\cdot)$ is $\beta$-smooth.
	
	\textbf{Strong convexity.} In order to establish the convexity properties of $f(\cdot)$, we consider the minimization problem
	\begin{equation*}
	\min \left\{ \langle \mathrm{hess} f(\Sigma) [V], V \rangle_{\Sigma} : V \in \mbb S^n,~\|V\|_\Sigma = 1 \right\},
	\end{equation*}
	which, by \eqref{eq:hess_eig} and the definition of $\|\cdot\|_\Sigma$, is equivalent to
	\[
	\min \left\{ \Tr{ \Sigma^{-1} S \Sigma^{-1} V \Sigma^{-1} V }: V \in \mbb S^n,~\tfrac{1}{2}\Tr{\Sigma^{-1} V \Sigma^{-1} V } = 1 \right\}.
	\]
	The optimal value of this problem is lower bounded by $ 2\, \lambda_{\min}(S) / \lambda_{\max}(\Sigma)$. Using the bound in Lemma~\ref{lemma:compact}\ref{lemma:compact_iii}, we have 
	\[
	\frac{2\, \lambda_{\min}(S)}{\lambda_{\max}(\Sigma)} = \frac{2 \lambda_{\min}(S)}{\lambda_{\max}(\widehat{\Sigma}) \exp(\sqrt{2}\rho)} = \sigma.
	\]
	Since $S \succ 0$, $\sigma > 0$. By Lemma~\ref{lem:F}\ref{lem:F_sc}, $f(\cdot)$ is thus $\sigma$-strongly convex. This completes the proof.
\end{proof}

\renewcommand\thesection{Appendix~\Alph{section}}
\section{Proofs of Section~\ref{sect:KL}}
\label{sect:proof:KL}
\renewcommand\thesection{\Alph{section}}
\setcounter{equation}{0}
\renewcommand{\theequation}{D.\arabic{equation}}

\begin{proof}[Proof of Theorem~\ref{thm:KL}]
	By applying the change of variables $\Z \leftarrow \cov^{-1}$, problem~\eqref{eq:KL:prob} can be reformulated as
	\begin{equation} \label{eq:KL:prob:X}
	\Inf{\Z} \left\{ \Tr{S \Z} - \log \det  \Z : \Z \succ 0, \, \Tr{\covsa \Z} - \log \det \Z  \leq \bar \radius\right\},
	\end{equation} 
	where $\bar\radius \Let 2 \radius + n + \log \det \covsa$.	Note that~\eqref{eq:KL:prob:X} is equivalent to
	\begin{align}
	& \Inf{\Z \succ 0} \Sup{\dualvar \geq 0} \; \Tr{S \Z} - \log\det \Z + \dualvar \left( \Tr{\covsa Z} - \log\det \Z - \bar \radius \right) \notag \\
	=& \Sup{\dualvar \geq 0} \Inf{\Z \succ 0} \; -\dualvar \bar \radius + \Tr{ (S + \dualvar \covsa)\Z} - (1+ \dualvar) \log \det \Z \notag \\
	=& \Sup{\dualvar \geq 0} \left\{ -\dualvar \bar \radius + \Inf{\Z \succ 0} \left\{ \Tr{(S + \dualvar \covsa) \Z} - (1 + \dualvar) \log \det \Z \right\} \right\}, \label{eq:KL:duality}
	\end{align}
	where the first equality follows from strong duality, which holds because $\radius > 0$ and because $\covsa^{-1}$ is a Slater point for the primal problem~\eqref{eq:KL:prob:X}.
	
	To analyze problem~\eqref{eq:KL:duality}, assume first that $S$ is singular. If $\dualvar = 0$, then the inner minimization problem over $\Z$ is unbounded, and thus $\dualvar = 0$ is never optimal for the outer maximization problem.	For any $\dualvar > 0$, the inner minimization problem over $\Z$ admits the optimal solution $\Z\opt(\dualvar) = (1 + \dualvar) (S + \dualvar \covsa)^{-1}$. Problem~\eqref{eq:KL:prob:X} is thus equivalent to
	\begin{align}
	& \Sup{\dualvar > 0} \Big\{ -\dualvar \bar \radius + n (1 + \dualvar) - (1 + \dualvar) \log\det[(1+\dualvar)(S + \dualvar \covsa)^{-1}] \Big\} \notag \\
	=& \Sup{\dualvar > 0} \Big\{ -\dualvar \bar \radius + n (1 + \dualvar) - (1 + \dualvar) \log[(1+\dualvar)^n \det(S + \dualvar \covsa)^{-1}] \Big\} \notag \\
	=& \Sup{\dualvar > 0} \Big\{ -\dualvar \bar \radius + n (1 + \dualvar) - n(1 + \dualvar) \log(1+\dualvar)  - (1 + \dualvar) \log \det (S + \dualvar \covsa)^{-1} \Big\} \label{eq:KL:refor}.
	\end{align}		
	By strong duality, any minimizer $\dualvar\opt$ of~\eqref{eq:KL:cov} can be used to construct a minimizer 
	$$\cov\opt = (1+\dualvar\opt)^{-1}(S + \dualvar\opt \covsa)$$
	for problem~\eqref{eq:KL:prob}.
	This observation establishes the claim if $S$ is singular. 
	
	Assume next that $S$ has full rank. In this case, the inner minimization problem in~\eqref{eq:KL:duality} admits the optimal solution $\Z\opt(\dualvar) = (1 + \dualvar) (S + \dualvar \covsa)^{-1}$ for any fixed $\dualvar \ge 0$, and thus problem~\eqref{eq:KL:duality} is equivalent to
	\[
	\Sup{\dualvar \geq 0} \Big\{ -\dualvar \bar \radius + n (1 + \dualvar) - n(1 + \dualvar) \log(1+\dualvar)  - (1 + \dualvar) \log \det (S + \dualvar \covsa)^{-1} \Big\},
	\]
	which differs from~\eqref{eq:KL:refor} only in that it has a closed feasible set, that is, $\dualvar = 0$ is feasible. Because the objective function of the above optimization problem is continuous in $\dualvar$, we can in fact optimize over $\dualvar > 0$ without reducing the supremum. The claim now follows by replacing $\bar \radius$ with its definition and eliminating the constant term from the objective function.
\end{proof}

\begin{proof}[Proof of Corollary~\ref{corol:singular}]
	For any $\covsa \in \PD^n$ and $\dualvar > 0$, the Woodbury formula~\cite[Corollary~2.8.8]{ref:Bernstein-2009} implies that 
	\[
	(S + \dualvar \covsa)^{-1} = \dualvar^{-1} \covsa^{-\half}(\dualvar^{-1}\covsa^{-\half} S \covsa^{-\half} + I_n)^{-1} \covsa^{-\half},
	\]
	and thus we have
	\begin{align*}
	\log\det (\dualvar \covsa+S)^{-1} + n \log \dualvar + \log\det \covsa &= -\log \det \left(I_n + \dualvar^{-1} \covsa^{-\half} \Chol \Chol^\top \covsa^{-\half} \right) \\
	&= -\log \det \left(I_k + \dualvar^{-1} \Chol^\top \covsa^{-1} \Chol \right) \\
	&= k\log \dualvar - \log\det \left(\dualvar I_k + \Chol^\top \covsa^{-1} \Chol \right),
	\end{align*}
	where the second equality follows from \cite[Equation~2.8.14]{ref:Bernstein-2009}. Substituting the expression for $\log\det(\dualvar \covsa + S)^{-1}$ into~\eqref{eq:KL:cov} and removing the irrelevant constant term $(n + \log\det \covsa)$ yields the equivalent minimization problem
	\begin{align*}
	\Inf{\dualvar > 0} \Big\{2 \dualvar \radius + n(1+\dualvar) \log \left(1+\dualvar\right) - (n-k)(1+\dualvar) \log\dualvar -(1+\dualvar) \log \det (\dualvar I_k +  \Chol^\top \covsa^{-1} \Chol ) \Big\}. 
	\end{align*}
	This observation completes the proof.
\end{proof}

\renewcommand\thesection{Appendix~\Alph{section}}
\section{Derivatives of Problem~\eqref{eq:KL:cov}} 
\label{sect:appendix:derivatives}
\renewcommand\thesection{\Alph{section}}
\setcounter{equation}{0}
\renewcommand{\theequation}{E.\arabic{equation}}

Use $g_1(\dualvar)$ as a shorthand for the objective function of problem~\eqref{eq:KL:cov}. In the following, we provide the first- and second-order derivatives of $g_1(\cdot)$, which are needed by the optimization algorithm that solves~\eqref{eq:KL:cov}. In particular, the first-order derivative is given by
\begin{align*}
g_1'(\dualvar) &= 2\radius + n\left( \log(1+\dualvar) + 1 \right) - \log\det \left( \dualvar I_n  + S \covsa^{-1} \right) - (1+\dualvar) \Tr{(\dualvar I_n + S \covsa^{-1})^{-1}},
\end{align*}
and the second-order derivative can be expressed as
\begin{align*}
&g_1''(\dualvar) =  \frac{n}{1+\dualvar} - \Tr{(\dualvar I_n + S \covsa^{-1})^{-1} \left( 2 I_n + (1+\dualvar)(\dualvar I_n + S \covsa^{-1} )^{-1} \right)}.
\end{align*}

Next, denote by $g_2$ the objective function of the singular reduction problem of Corollary~\ref{corol:singular}, that is,
\[ g_2(\dualvar) = 2\dualvar \radius + n(1+\dualvar) \log \left(1+\dualvar\right) - (n-k)(1+\dualvar) \log\dualvar -(1+\dualvar) \log \det (\dualvar I_k +  \Chol^\top \covsa^{-1} \Chol ). \]
The first- and second-order derivative of $g_2$ are given by
\begin{align*}
g_2'(\dualvar) &= 2\radius + n\left( \log(1+\dualvar) + 1 \right) - (n-k) \left( \log \dualvar + \dualvar^{-1} + 1 \right) \\
&\qquad - \log\det \left( \dualvar I_k  + \Chol^\top \covsa^{-1} \Chol \right) - (1+\dualvar) \Tr{(\dualvar I_k + \Chol^\top \covsa^{-1} \Chol)^{-1}},
\end{align*}
and 
\begin{align*}
g_2''(\dualvar) =  &\frac{n}{1+\dualvar} - (n-k) (\dualvar^{-1} - \dualvar^{-2}) \\
&\mspace{45mu} - \Tr{\dualvar I_k + \Chol^\top \covsa^{-1} \Chol)^{-1} \left( 2 I_k + (1+\dualvar)(\dualvar I_k + \Chol^\top \covsa^{-1} \Chol)^{-1} \right)},
\end{align*}
respectively.

\paragraph{\bf Acknowledgments}
We gratefully acknowledge financial support from the Swiss National Science Foundation under grant BSCGI0\_157733 as well as the EPSRC grants EP/M028240/1, EP/M027856/1 and EP/N020030/1.

\bibliographystyle{abbrv}
\bibliography{bibliography}

\end{document}